\documentclass[pdflatex,sn-mathphys-num]{sn-jnl}

\usepackage{graphicx}%
\usepackage{multirow}%
\usepackage{amsmath,amssymb,amsfonts}%
\usepackage{amsthm}%
\usepackage{mathrsfs}%
\usepackage[title]{appendix}%
\usepackage{xcolor}%
\usepackage{textcomp}%
\usepackage{manyfoot}%
\usepackage{booktabs}%
\usepackage{algorithm}%
\usepackage{algorithmicx}%
\usepackage{algpseudocode}%
\usepackage{listings}%
\usepackage{caption}
\usepackage{subcaption}

\catcode`@=11
\@addtoreset{equation}{section}
\renewcommand\theequation{\thesection.\@arabic\c@equation}
\catcode`@=12

\theoremstyle{thmstyleone}%
\newtheorem{theorem}{Theorem}[section]

\theoremstyle{thmstyletwo}%
\newtheorem{example}{Example}%
\newtheorem{remark}{Remark}%

\theoremstyle{thmstylethree}%
\newtheorem{lemma}{Lemma}
\newtheorem{corollary}{Corollary}

\begin{document}

\title[An efficient spline-based scheme on Shishkin-type meshes]{\textbf{An efficient spline-based scheme on Shishkin-type meshes for solving singularly perturbed coupled systems with Robin boundary conditions}}

\author*[1]{\fnm{Kousalya} \sur{Ramanujam}}\email{ramanujamkousalya@gmail.com}

\author[1]{\fnm{Vembu} \sur{Shanthi}}\email{vshanthi@nitt.edu}

\affil*[1]{\orgdiv{Department of Mathematics}, \orgname{National Institute of Technology}, \orgaddress{ \city{Tiruchirappalli}, \postcode{620015}, \state{Tamil Nadu}, \country{India}}}

\abstract{In this paper, we investigate a weakly coupled system of singularly perturbed linear reaction-diffusion equations with Robin boundary conditions, where the leading terms are multiplied by small positive parameters that may differ in magnitude. The solution to this system exhibits overlapping and interacting boundary layers. To appropriately resolve these layers, we employ a standard Shishkin mesh and propose a novel modification of Bakhvalov-Shishkin mesh. A cubic spline approximation is applied at the boundary conditions, while a central difference scheme is used at the interior points. The problem is then solved on both meshes. It is demonstrated that the proposed scheme achieves almost second-order convergence upto a logarithmic factor on the Shishkin mesh and exact second-order convergence on the modified Bakhvalov-Shishkin mesh. We present numerical results to validate the accuracy of our findings.}



\maketitle


\section{Introduction and model problem} \label{S: introduction}
 Singularly perturbed differential equations model multiscale phenomena characterized by sharp gradients due to the presence of small positive parameters \cite{miller2012fitted, Farrell2000, roos2008robust}. These equations frequently arise in science and engineering contexts, including fluid dynamics, semiconductor theory, chemical kinetics and turbulence modeling. The solutions to such problems often exhibit steep gradients confined to narrow regions known as boundary layers. Classical numerical methods typically fail to capture these sharp transitions accurately. Semiconductor device modeling provides a prominent example. \citet{markowich1984singular} conducted a singular perturbation analysis of the fundamental semiconductor device equations under mixed Neumann-Dirichlet boundary conditions, identifying Debye length as the perturbation parameter. Further, \citet{Markowich1984} investigated a one-dimensional modelling of a biased semiconductor device. Because of the different orders of magnitude of the solution components at the boundaries, they scale the components individually and obtain a singular perturbation problem. In a related study, \citet{brezzi2006singular} treated reverse-biased semiconductor devices as a singular perturbation problem, where the perturbation problem is related to the temperature.

Singular perturbation techniques also extend to turbulence modeling in hydraulics \cite{sreenivasan2025turbulence}. The classical $k$--$\varepsilon$ two-equation turbulence model remains a standard approach \cite{rodi2017turbulence} to capture the effect of turbulence. Building on this, \citet{thomas1998towards} developed a hierarchy of systems using such models, with each system reducing to a set of singularly perturbed second-order boundary value problems. Further applications of singularly perturbed differential equations can be found in \cite{hegarty1993key, naidu2001singular, kazakov2006computational, shafique2016boundary, turner2018elementary, lalrinhlua2025study, kumar2024impact, singh2021}. Among the wide variety of singular perturbation problems, those involving Robin-type boundary conditions have drawn considerable attention \cite{ansari2003numerical, cai2004uniform, das2012higher, selvi2017parameter, gupta2023higher, gelu2024efficient, saini2024parameter}. Motivated by the work of \citet{kumar2024impact} on combustion models involving mixed-type flux, this study investigates a class of weakly coupled linear second-order singularly perturbed systems. To address this, we develop a cubic-spline based method on Shishkin-type meshes and provide a rigorous convergence analysis demonstrating its effectiveness. The system under study is as follows:

\begin{equation}  \label{P: system general}
	L \vec{y} \equiv
	\begin{pmatrix}
		L_1 \vec{y}  \\
		L_2 \vec{y}
	\end{pmatrix} \equiv
	\begin{pmatrix}
		-\varepsilon^2 \dfrac{d^2}{dx^2} & 0                       \\
		0                               & -\mu^2 \dfrac{d^2}{dx^2}
	\end{pmatrix} \vec{y}
	+ B \vec{y}
	= \vec{f} \qquad \mbox{for}~ x \in \Omega := (0,1),
\end{equation}

\noindent where $0 < \varepsilon \leqslant \mu \leqslant 1$ (without loss of generality),

\begin{equation}  \label{P: coefficient matrix and source term}
	B = 
	\begin{pmatrix}
		b_{11}(x) & b_{12}(x) \\
		b_{21}(x) & b_{22}(x)
	\end{pmatrix}
	\mbox{~and~}
	\vec{f}(x) =
	\begin{pmatrix}
		f_1(x) \\
		f_2(x)
	\end{pmatrix}.
\end{equation}


\noindent The Robin boundary conditions at $ \partial \Omega$  are given by
\begin{equation} \label{P: robin boundary conditions}
	\begin{aligned}
		R_1^{(0)}  y_1(0) &\equiv \alpha_1 y_1(0) - \varepsilon \beta_1 y_1'(0) = P_1, \quad
		R_1^{(1)}  y_1(1) \equiv \gamma_1 y_1(1) + \varepsilon \delta_1 y_1'(1) = Q_1, \\
		R_2^{(0)}  y_2(0) &\equiv \alpha_2 y_2(0) - \mu \beta_2 y_2'(0) = P_2, \quad
		R_2^{(1)}  y_2(1) \equiv \gamma_2 y_2(1) + \mu \delta_2 y_2'(1) = Q_2.
	\end{aligned}
\end{equation}

\noindent where $\alpha_i, \beta_i \geqslant 0, \alpha_i + \beta_i > 0, \gamma_i > 0$ and $\delta_i \geqslant 0$ for $i = 1,2$. The solution of \eqref{P: system general}~--~\eqref{P: robin boundary conditions} is the vector $\vec{y}(x):=\left( y_1(x), y_2(x)\right)^T.$
For brevity, we define ${R}^{(k)} \vec{y}(k) = (R_1^{(k)} {y}_1(k), ~ R_2^{(k)} {y}_2(k) )^T$, $P =(P_1, P_2)^T$ and $Q = (Q_1,    Q_2)^T$ for $k=0,1.$ Assume that for all $x \in \Bar{\Omega}$, the following holds:
\begin{equation}  \label{P: assumption on coefficient matrix negative}
	b_{12}(x) \leqslant 0, 
	\qquad  
	b_{21}(x) \leqslant 0 
\end{equation}
and there exists a constant $\lambda > 0$ such that
\begin{equation}  \label{P: assumption on coefficient matrix positive}
	\lambda^2 
	< \min \left\{ 
	b_{11}(x) + b_{12}(x), ~~ 
	b_{21}(x) + b_{22}(x) 
	\right\} \quad \text{for all} \quad x \in \Bar{\Omega}.
\end{equation}
These imply that $b_{11}(x) > | b_{12}(x) |$ and  $b_{22}(x) >  | b_{21}(x) |$ for every $x$ in $\Bar{\Omega}.$
The functions $b_{21}, b_{22}, f_1,f_2 \in \mathcal{C}^{{\,}2} (\Bar{\Omega})$ and $b_{11}, b_{12} \in \mathcal{C}^{{\,}3}(\Bar{\Omega}).$  Assume $|b_{ij}(x) | \leqslant \lambda^* $ for $1 \leqslant i,j \leqslant 2$. The derivative bounds of these functions are independent of the perturbation parameters $\varepsilon$ and $\mu$. Regarding the existence and uniqueness of the solution to \eqref{P: system general}~--~\eqref{P: robin boundary conditions}, see Section~\ref{S: a priori bounds on the solution and its derivatives}. Depending on the relation between $\varepsilon$ and $\mu,$ the following three cases can be considered:
\begin{equation*}
	\text{(i)} ~\varepsilon = \mu \in (0,1], \qquad \text{(ii)} ~ \varepsilon \in (0,1], ~\mu = 1, \qquad \text{and} \qquad \text{(iii)}~ 0 < \varepsilon \leqslant \mu \leqslant 1.
\end{equation*}

\citet{matthews2000parameter} analyzed system~\eqref{P: system general} under case~(i) with Dirichlet boundary conditions and established first-order convergence. This result was subsequently improved by \citet{linss2003improved}, who obtained second-order convergence. More recently, \citet{kaushik2024adaptive} attained second-order convergence for case~(i) using a layer-adaptive mesh. For case~(ii), \citet{matthews2002numerical} obtained nearly second-order convergence. Case~(iii) was examined in \citet{madden2003uniformly}, demonstrating nearly first-order accuracy, and further enhanced by \citet{linss2009layer}, who achieved second-order convergence on both Bakhvalov and equidistribution meshes. 
 
\citet{priyadharshini2013uniformly} proposed hybrid schemes on the Shishkin mesh for convection-diffusion systems subject to mixed-type boundary conditions, achieving nearly second-order convergence. \citet{das2013uniformly} applied a hybrid cubic spline method to \eqref{P: system general}~--~\eqref{P: robin boundary conditions} for case (i) and obtained almost second-order convergence. \citet{das2020higher} achieved second-order convergence for case (iii) using equidistributed meshes. In this work, we establish second-order convergence for the problem \eqref{P: system general}~--~\eqref{P: robin boundary conditions} in the most general case (case (iii)) using a modified Bakhvalov~--~Shishkin mesh (BS mesh). Our approach avoids the complexities associated with equidistributed meshes and still achieves second-order convergence on the relatively simple BS mesh. We also compare our results with those obtained using the standard Shishkin mesh and demonstrate that the proposed method on the modified BS mesh is more effective.

	Section~\ref{S: a priori bounds on the solution and its derivatives} discusses the maximum principle, stability results and bounds on the solution and its derivatives. Section~\ref{S: discrete problem} introduces the meshes and finite-difference operator used to approximate \eqref{P: system general}~--~\eqref{P: robin boundary conditions}. Section~\ref{S: error analysis} demonstrates that the proposed scheme achieves almost second-order convergence on standard Shishkin mesh (S~--~mesh) and second-order convergence on a modified Bakhvalov-Shishkin mesh (BS~--~mesh) in the maximum norm. Numerical experiments validating these results are provided in Section~\ref{S: numerical results}.

The symbol $C$ will represent a generic constant that may vary from line to line but remains independent of $\varepsilon, \mu$ and the mesh. When a constant is denoted with a subscript (such as $C_1$), it refers to a specific, fixed value that does not change and is  independent of the perturbation parameters and the mesh.


\section{Analytical properties of the exact solution} \label{S: a priori bounds on the solution and its derivatives}

This section analyzes the continuous solution by deriving bounds on both the solution and its derivatives. The solution is split into smooth and layer components, with derivative bounds obtained for each. For any $\vec{y} \in \mathcal{C}^{{\,}0}(\Bar{\Omega}) \cap  \mathcal{C}^{{\,}2} (\Omega),$ the differential operator $L$ in \eqref{P: system general} satisfies the following maximum principle.


\begin{lemma}[Maximum principle]  \label{L: maximum principle}
	Assuming \eqref{P: assumption on coefficient matrix negative}~--~\eqref{P: assumption on coefficient matrix positive},  if ${  R}^{(0)} \vec{y}(0) \geqslant 0$,  ${  R}^{(1)} \vec{y}(1) \geqslant 0$ ~and~ $L \vec{y}(x) \geqslant~ 0$ for all $x \in \Omega$, then $\vec{y}(x) \geqslant 0$ for all $x \in \Bar{\Omega}.$
\end{lemma}

\begin{proof} 
	~Define $\vec{t}(x)= \left( 2-x, 2-x \right)^T,$ which satisfies $\vec{t}(x) > 0$ and $L \vec{t}(x) > 0$ for all $x \in \Bar{\Omega}$, along with the boundary conditions 
	\begin{align}  \label{boundary conditions for t}
		R^{(0)} \vec{t}(0) > 0 
		\quad \text{and} \quad  
		R^{(1)} \vec{t}(1) > 0.
	\end{align}
	\noindent	Assuming the theorem is false, define $\eta = \max \left\{ \max_{x \in \Bar{\Omega}} \left( {-y_1}/{t_1} \right), \max_{x \in \Bar{\Omega}} \left( {-y_2}/{t_2} \right) \right\}. $ Also, 
	
	\begin{align}  \label{Eq: y + eta t > 0}
		\vec{y} + \eta \vec{t}(x) \geqslant 0 
		\quad \text{for all} \quad x \in \Bar{\Omega}.
	\end{align}

	\noindent Since $\eta>0,$ choose $x^*$ such that $\left( -y_1 / t_1 \right)(x^*) = \eta$ or $\left( -y_2 / t_2 \right)(x^*) = \eta.$ 
	
	\noindent	\textit{Case 1.} If $\left( y_1 + \eta t_1 \right)(x^*) = 0$ at $x^* = 0,$ then by \eqref{Eq: y + eta t > 0}, $y_1 + \eta t_1$ attains a minimum at $x^*,$ implying $  R_1^{(0)} \left( y_1 + \eta t_1 \right)(x^*) = 0.$ However, using the hypothesis and \eqref{boundary conditions for t}, we obtain $  R_1^{(0)} \left( y_1 + \eta t_1 \right)(x^*) > 0,$ leading to a contradiction.
	
	\noindent \textit{Case 2.} If $\left( y_1 + \eta t_1 \right)(x^*) = 0$ at $x^* = 1,$ the argument follows as in Case 1, leading to a contradiction.
	\noindent \textit{Case 3.} Suppose $\left( y_1 + \eta t_1 \right)(x^*) = 0$ holds for some $x^* \in \Omega,$ the hypothesis gives $L_1 ( \vec{y} + \eta \vec{t}{\,}) (x^*) > 0,$ while \eqref{Eq: y + eta t > 0} implies $L_1 ( \vec{y} + \eta \vec{t}{\,}) (x^*) \leqslant 0,$  a contradiction. \\
	Applying the same reasoning to $y_2 + \eta t_2$ completes the proof, showing $\vec{y}(x) \geqslant 0$	for all $x \in \Bar{\Omega}.$     \hfill
\end{proof}


A fundamental property of the operator $L$ corresponding to the coupled system \eqref{P: system general} is presented in the following Lemma.
\begin{lemma}[Comparison principle]  \label{L: comparison principle}
	If $  R^{(0)} \vec{y}(0) \geqslant |   R^{(0)} \vec{z}(0) |$, $  R^{(1)} \vec{y}(1) \geqslant |   R^{(1)} \vec{z}(1) |$ and $L \vec{y} \geqslant | L \vec{z} | $ on $\Omega,$ then $\vec{y}(x) \geqslant | \vec{z}(x)| $ for all $x \in \Bar{\Omega}.$
\end{lemma}

\begin{proof}
	~	Let $\vec{u}(x) = \vec{y}(x) - |\vec{z}(x)|.$ Applying Lemma~\ref{L: maximum principle} to $\vec{u},$ the desired result follows.
\end{proof}

	We call $\vec{y}$ a barrier function for $\vec{z},$ in the context of Lemma~\ref{L: comparison principle}. Under assumptions \eqref{P: assumption on coefficient matrix negative}~--~\eqref{P: assumption on coefficient matrix positive}, the reaction-diffusion system \eqref{P: system general} admits a unique classical solution $\vec{y} \in \mathcal{C}^4 (\Bar{\Omega}) \times \mathcal{C}^4 (\Bar{\Omega}).$ This result follows from Lemma~\ref{L: comparison principle} and standard analytical techniques (cf.\ \citet{ladyzhenskaya1968linear} ).

\begin{corollary}[Stability estimate]   \label{L: stability result} 
	If $\vec{y}$ is the solution of \eqref{P: system general}~--~\eqref{P: assumption on coefficient matrix positive}, 
	then $\vec{y}(x)$ satisfies the following stability bound:
	\begin{equation*}
		\| \vec{y}{\,} \| \leqslant 
		C \max \left\{ 
		\left \| R^{(0)} \vec{y}(0) \right \|, 
		\left \| R^{(1)} \vec{y}(1) \right \|, 
		\dfrac{1}{\lambda^2} \| L \vec{y}{\,} \|
		\right\}.
	\end{equation*}	
\end{corollary}


\begin{proof}
	~ Define the barrier functions
	\begin{align*}
		\vec{\psi}^\pm(x) 
		= C \max \left\{ 
		\left\| R^{(0)} \vec{y}(0) \right\|,\,
		\left\| R^{(1)} \vec{y}(1) \right\|,\,
		\dfrac{1}{\lambda^2} \| L \vec{y}{\,} \|
		\right\} 
		\begin{pmatrix}
			1 \\
			1
		\end{pmatrix}
		\pm \vec{y}(x), \quad x \in \Bar{\Omega}.
	\end{align*}
	
	It is straightforward to verify that ${  R}^{(0)} \vec{\psi}^\pm(0) \geqslant 0, ~ 
	{  R}^{(1)} \vec{\psi}^\pm(1) \geqslant 0$ and $L \vec{\psi}^\pm(x) \geqslant 0$ on $\Omega.$ Applying Lemma~\ref{L: maximum principle} yields $\vec{\psi}^\pm \geqslant 0$ on $\Bar{\Omega}$ and the required result follows.       
\end{proof}


\noindent	A bound on the solution $\vec{y}$ of the Problem~\eqref{P: system general}~--~\eqref{P: robin boundary conditions} and its derivatives, is given in the following Lemma.
\begin{lemma}[Bounds for the solution derivatives]  \label{L: bounds for the solution derivatives}
	Let $\vec{y}=(y_1,y_2)^T$ be the solution of \eqref{P: system general}~--~\eqref{P: robin boundary conditions}. Then for $k=1,2,$	
	\begin{alignat}{2}
		\| y_1^{(k)} \|       &\leq C ( 1 + \varepsilon^{-k} ), \quad 
		&\quad \| y_2^{(k)} \|       &\leq C ( 1 + \mu^{-k} ), \label{Eq: bounds on solution and derivatives 1} \\[4pt]
		\| y_1^{(k+2)} \|     &\leq C \varepsilon^{-2} ( \varepsilon^{-k} + \mu^{-k} ), \quad  
		&\quad \| y_2^{(k+2)} \|     &\leq C \mu^{-2} ( \varepsilon^{-k} + \mu^{-k} ). \notag
	\end{alignat}
	
\end{lemma}

	\begin{proof}
	~ The case $k=1$ in \eqref{Eq: bounds on solution and derivatives 1} is proved in \citet{miller2012fitted}. The other bounds follow from \eqref{Eq: bounds on solution and derivatives 1} and Lemma~\ref{L: stability result}.      
	\hfill
\end{proof}


In order to derive sharper bounds needed for the error estimates, we decompose the solution into \lq smooth\rq~ and \lq layer\rq~ components. This is achieved by formulating auxiliary problems for each part, as proposed in \citet{das2020higher}, based on the original system \eqref{P: system general}~--~\eqref{P: robin boundary conditions}. Let

\begin{equation*}  \label{Eq: solution decomposition}
	\vec{y} = \vec{v} + \vec{w},
\end{equation*}
where $\vec{v}=\left( v_1, v_2 \right)^T$ is the solution to the problem
\begin{align}   \label{Eq: solution decomposition smooth part}
	L \vec{v} = \vec{f} ~  \text{on} ~ \Omega, \quad
	{R}^{(0)} \vec{v}(0) = B(0)^{-1} \vec{f}(0), \quad {R}^{(1)} \vec{v}(1) = B(1)^{-1} \vec{f}(1),
\end{align}
and $\vec{w} = \left( w_1, w_2 \right)^T$ is the solution to the problem
\begin{align}   \label{Eq: solution decomposition layer part}
	L \vec{w} = \vec{0} ~ \text{on} ~ \Omega, \quad {R}^{(0)} \vec{w}(0) =   P -{  R}^{(0)} \vec{v}(0), \quad {R}^{(1)} \vec{w}(1) =   Q -{  R}^{(1)} \vec{v}(1).
\end{align}
Let us define the following layer functions for analysis of the layer part
	\begin{align*}  
		\mathbb{B}_{\varepsilon}(x) &= \exp \left( -x \lambda / \varepsilon \right) + \exp \left( -(1-x) \lambda / \varepsilon \right), \\
		\mathbb{B}_{\mu}(x)         &= \exp \left( -x \lambda / \mu \right) + \exp \left( -(1-x) \lambda / \mu \right).
	\end{align*}


\noindent The lemma below gives derivative bounds for the smooth \eqref{Eq: solution decomposition smooth part} and layer components \eqref{Eq: solution decomposition layer part} of the solution $\vec{y}$.
\begin{lemma}  \label{L: derivative estimates of smooth and layer components}
	For all $x \in \Bar{\Omega},$ the smooth component $\vec{v}$ defined in \eqref{Eq: solution decomposition smooth part} satisfies:
	\begin{align}  \label{Eq: v1 v2 k=0 to 4}
		& \| v_1^{(k)} \| \leqslant C ( 1 + \varepsilon^{2-k} ), \quad 
		\| v_2^{(k)} \| \leqslant C ( 1 + \mu^{2-k} ) 
		\quad \text{for}~k = 0, \ldots, 4.
	\end{align}
	and
	\begin{equation*}
		| v_1^{(3)} (x) | \leqslant C \left( \mu^{-1} + \varepsilon^{-1} \mathbb{B}_\varepsilon(x) \right) \quad \text{for} ~ x \in (0,1). 
	\end{equation*}
	The layer component $\vec{w}$ satisfies:
	\begin{align}   \label{Eq: derivative bounds for w1 w2 for k=0 to 2}
		| w_1^{(k)} (x)  | \leqslant C ( \varepsilon^{-k} \mathbb{B}_\varepsilon(x) + \mu^{-k} \mathbb{B}_\mu(x) ), \quad
		| w_2^{(k)} (x)  | \leqslant C \mu^{-k} \mathbb{B}_\mu(x) \quad \text{for}~k=0,1,2,
	\end{align}
	and
	\begin{align}  \label{Eq: w1 w2 k=3,4}
		\varepsilon^2  | w_1^{(k)} (x) | + \mu^2  | w_2^{(k)} (x) | \leqslant
		C ( \varepsilon^{2-k} \mathbb{B}_\varepsilon(x)  + \mu^{2-k} \mathbb{B}_\mu(x) ) \quad \text{for}~k=3,4.
	\end{align}
\end{lemma}
\begin{proof}
	~ The result has been proven in \citet{linss2004accurate}.
	\hfill
\end{proof}

	The layer term defined in \eqref{Eq: solution decomposition layer part} can be further decomposed as follows:
\begin{lemma}   \label{L: w derivatives}
	Assume that $\varepsilon \leqslant \mu$. Then, the layer component $\vec{w}$ admits two distinct decompositions: 
	\begin{equation*}
		\vec{w} = \vec{w}_\varepsilon + \vec{w}_\mu = \vec{\hat{w}}_\varepsilon + \vec{\hat{w}}_\mu,
	\end{equation*}
	where $( \vec{w}_\varepsilon, \vec{w}_\mu )$ and $( \vec{\hat{w}}_\varepsilon, \vec{\hat{w}}_\mu )$ satisfy the following derivative estimates for all $x \in [0,1]$:
	\begin{align*}
		\varepsilon^2 | w{\,}''_{\varepsilon,1}(x) | + \mu^2 | w{\,}''_{\varepsilon,2}(x) | 
		&\leqslant C{\,} \mathbb{B}_\varepsilon(x), \\
		| w^{\,(3)}_{\mu,1}(x) | + | w^{\,(3)}_{\mu,2}(x) | 
		&\leqslant C \mu^{-3} \mathbb{B}_\mu(x),   \\
		\varepsilon^2 | \hat{w}{\,}''_{\varepsilon,1}(x) | + \mu^2 | \hat{w}{\,}''_{\varepsilon,2}(x) | 
		&\leqslant C{\,} \mathbb{B}_\varepsilon(x),  \\
		\varepsilon^2 | \hat{w}^{\,(4)}_{\mu,1}(x) | + \mu^2 | \hat{w}^{\,(4)}_{\mu,2}(x) | 
		&\leqslant C \mu^{-2} \mathbb{B}_\mu(x).  
	\end{align*}
\end{lemma}

\begin{proof}
	~	For the complete proof, refer to \citet{linss2004accurate} and \citet{das2013uniformly}.
	\hfill
\end{proof}

	
\section{Discrete problem}  \label{S: discrete problem}
To approximate the solution of problem~\eqref{P: system general}~--~\eqref{P: robin boundary conditions}, we employ a finite difference scheme (standard central difference and cubic spline based) defined on either standard Shishkin mesh ($\Omega_S$) or modified BS mesh ($\Omega_{BS}$). Let $\vec{Y} = (Y_1, Y_2)^T$ denote the mesh function defined on $\Omega^N = \{ x_i \}_0^N$ that satisfies the discrete system
\begin{equation}  \label{P: discretization system general}
	{L}^N \vec{Y}(x_i) \equiv
	\begin{pmatrix}
		L_1^N \vec{Y}(x_i)  \\
		L_2^N \vec{Y}(x_i)
	\end{pmatrix} \equiv
	\begin{pmatrix}
		-\varepsilon^2 \delta^2 & 0                       \\
		0                               & -\mu^2 \delta^2
	\end{pmatrix} \vec{Y}(x_i)
	+ B(x_i) \vec{Y}(x_i)
	= \vec{f}(x_i) \quad \mbox{for}~ i=1,\dotsc,N-1,
\end{equation}
with Robin boundary conditions
\begin{equation}
		\begin{aligned}  \label{P: discretization robin boundary conditions}
			R_1^{(0),N}  {Y}_1(x_0) &\equiv \alpha_1 Y_1(x_0) - {\varepsilon}\beta_1 S^+ Y_1(x_0) =   P_1, \quad
			R_1^{(1),N}  {Y}_1(x_N) \equiv \gamma_1 Y_1(x_N) + {\varepsilon}\delta_1 S^-Y_1(x_N) =   Q_1, \\
			R_2^{(0),N}  {Y}_2(x_0) &\equiv \alpha_2 Y_2(x_0) - {\mu}\beta_2 S^+Y_2(x_0) =   P_2, \quad
			R_2^{(1),N}  {Y}_2(x_N) \equiv \gamma_2 Y_2(x_N) + {\mu}\delta_2 S^-Y_2(x_N) =   Q_2,
		\end{aligned}
\end{equation}
where $\delta^2$ is the standard second-order central differencing operator. To discretize the Robin boundary conditions \eqref{P: discretization robin boundary conditions}, we use one-sided derivatives based on the cubic spline interpolant $S(x)$ constructed as in \citet{bawa2010higher}, is detailed below.
\begin{align*}
	S^-(x_i) & = \; \dfrac{h_i}{6} M_{i-1} + \dfrac{h_i}{3} M_{i} + \dfrac{y(x_i) - y(x_{i-1})}{h_i},  \\
	S^+(x_i) & = \; \dfrac{-h_{i+1}}{3} M_{i} - \dfrac{h_{i+1}}{6} M_{i+1} + \dfrac{y(x_{i+1}) - y(x_{i})}{h_{i+1}}.  
\end{align*}
Here $h_i = x_i - x_{i-1}$ represents the step size and we impose that $M_i = y_j''(x_i)$ for $j=1,2$.  After discretization, the coefficients will be as follows:
\begin{equation}
	\begin{aligned}  \label{Eq: discretization system general}
			L_1^N \vec{Y} & \equiv A_{1,i}^- Y_{1,i-1} + A_{1,i}^c Y_{1,i} + A_{1,i}^+ Y_{1,i+1} + B_{1,i}^- Y_{2,i-1} + B_{1,i}^c Y_{2,i} + B_{1,i}^+ Y_{2,i+1}, \\	
			L_2^N \vec{Y} & \equiv A_{2,i}^- Y_{2,i-1} + A_{2,i}^c Y_{2,i} + A_{2,i}^+ Y_{2,i+1} + B_{2,i}^- Y_{1,i-1} + B_{2,i}^c Y_{1,i} + B_{2,i}^+ Y_{1,i+1}, \\
			F_1^N & \equiv F_{1,i}^- f_{1,i-1} + F_{1,i}^c f_{1,i} + F_{1,i}^+ f_{1,i+1}, \\
			F_2^N & \equiv F_{2,i}^- f_{2,i-1} + F_{2,i}^c f_{2,i} + F_{2,i}^+ f_{2,i+1}.
	\end{aligned}
\end{equation}

\noindent The computation of the coefficients is outlined below.  For $i=0,$ the discretization takes the form:
\begin{equation}   \label{Eq: left boundary conditions - cubic spline}
	\begin{cases}
		A_{1,0}^c = \dfrac{3 \varepsilon}{h_1} \left( \alpha_1 + \dfrac{\varepsilon \beta_1}{h_1} \right) + b_{11}(x_0)\beta_1; \quad  
		A_{1,0}^+ = \dfrac{-3 \varepsilon^2 \beta_1}{h_1^2} + \dfrac{b_{11}(x_1) \beta_1}{2};  \\[4pt]
		B_{1,0}^c = \beta_1 b_{12}(x_0); \quad  
		B_{1,0}^+ =  \dfrac{\beta_1 b_{12}(x_1)}{2}; \\[4pt]
		F_{1,0}^- = \dfrac{3 \varepsilon   P_1}{h_1}; \quad  
		F_{1,0}^c = \beta_1; \quad  
		F_{1,0}^+ =  \dfrac{\beta_1}{2}; \\[4pt]                    
		A_{2,0}^c = \dfrac{3 \mu}{h_1} \left( \alpha_2 + \dfrac{\mu \beta_2}{h_1} \right) + b_{22}(x_0)\beta_2; \quad  
		A_{2,0}^+ = \dfrac{-3 \mu^2 \beta_2}{h_1^2} + \dfrac{b_{22}(x_1) \beta_2}{2}; \\[4pt]
		B_{2,0}^c = \beta_2 b_{21}(x_0); \quad  
		B_{2,0}^+ = \dfrac{\beta_2 b_{21}(x_1)}{2}; \\[4pt] 
		F_{2,0}^- = \dfrac{3 \mu   P_2}{h_1}; \quad  
		F_{2,0}^c = \beta_2; \quad  
		F_{2,0}^+ = \dfrac{\beta_2}{2}.
	\end{cases}
\end{equation}

\noindent For $i=1, \ldots, N-1,$ the discretization is as follows:   
\begin{equation}   \label{Eq: discretization central scheme}
	\begin{cases}
		A_{1,i}^- = \dfrac{- \varepsilon^2}{h_i \bar{h}_i}; \quad  
		A_{1,i}^c = \dfrac{2 \varepsilon^2}{h_i h_{i+1}} + c_{11}(x_i); \quad  
		A_{1,i}^+ = \dfrac{- \varepsilon^2}{h_{i+1} \bar{h}_i};  \\[4pt]
		B_{1,i}^- = 0; \quad  
		B_{1,i}^c = c_{12}(x_i); \quad  
		B_{1,i}^+ = 0; \\[4pt]
		F_{1,i}^- = 0; \quad  
		F_{1,i}^c = 1; \quad  
		F_{1,i}^+ = 0; \\[4pt]
		A_{2,i}^- = \dfrac{- \mu^2}{h_i \bar{h}_i}; \quad  
		A_{2,i}^c = \dfrac{2 \mu^2}{h_i h_{i+1}} + c_{22}(x_i); \quad  
		A_{2,i}^+ = \dfrac{- \mu^2}{h_{i+1} \bar{h}_i}; \\[4pt]
		B_{2,i}^- = 0; \quad  
		B_{2,i}^c = c_{21}(x_i); \quad  
		B_{2,i}^+ = 0; \\[4pt]
		F_{2,i}^- = 0; \quad  
		F_{2,i}^c = 1; \quad  
		F_{2,i}^+ = 0.
	\end{cases}
\end{equation}

\noindent For $i=N,$ the discretization is as follows: 
\begin{equation}  \label{Eq: right boundary conditions - cubic spline}
	\begin{cases} 
		A_{1,N}^- = \dfrac{-3 \varepsilon^2 \delta_1}{h_N^2} + \dfrac{b_{11}(x_{N-1}) \delta_1}{2}; \quad  
		A_{1,N}^c = \dfrac{3 \varepsilon}{h_N} \left( \gamma_1 + \dfrac{ \varepsilon \delta_1}{h_N} \right) + b_{11}(x_N)\delta_1;  \\[4pt]
		B_{1,N}^- = \dfrac{\delta_1 b_{12}(x_{N-1})}{2} ; \quad  
		B_{1,N}^c = b_{12}(x_N) \delta_1; \\[4pt]
		F_{1,N}^- = \dfrac{\delta_1}{2}; \quad
		F_{1,N}^c = \delta_1; \quad
		F_{1,N}^+ = \dfrac{3 \varepsilon   Q_1}{h_N}; \\[4pt]
		A_{2,N}^- = \dfrac{-3 \mu^2 \delta_2}{h_N^2} + \dfrac{b_{22}(x_{N-1}) \delta_2}{2}; \quad  
		A_{2,N}^c = \dfrac{3 \mu}{h_N} \left( \gamma_2 + \dfrac{\mu \delta_2}{h_N} \right) + b_{22}(x_N)\delta_2; \\[4pt]
		B_{2,N}^- = \dfrac{\delta_2 b_{21}(x_{N-1})}{2}; \quad  
		B_{2,N}^c = b_{21}(x_N)  \delta_2; \\[4pt]
		F_{2,N}^- = \dfrac{\delta_2}{2}; \quad  
		F_{2,N}^c = \delta_2; \quad
		F_{2,N}^+ = \dfrac{3 \mu   Q_2}{h_N};
	\end{cases}
\end{equation}

\noindent We now proceed to describe the meshes.

\textit{Shishkin mesh (S~--~mesh $\Omega_S$)}:
Let $N$ be a positive integer and a multiple of 2. The transition parameters $\tau_\varepsilon$ and $\tau_\mu$ are defined as
\begin{equation*}  \label{Eq: transition parameters}
	\tau_\mu = \min \left\{ \dfrac{1}{4},~ \dfrac{\sigma \mu}{\lambda} \ln N \right\} \quad \text{and} \quad \tau_\varepsilon = \min \left\{ \dfrac{1}{8},~ \dfrac{\tau_\mu}{2}, ~ \dfrac{\sigma \varepsilon}{\lambda}  \ln N \right\}.
\end{equation*}   
A piecewise-uniform mesh $\Omega_S^N = \{ x_i \}_0^N$ is constructed by dividing $[0,1]$ into five subintervals $[0, \tau_\varepsilon], [\tau_\varepsilon, \tau_\mu], [\tau_\mu, 1-\tau_\mu], [1-\tau_\mu, 1-\tau_\varepsilon]$ and $[1-\tau_\varepsilon,1].$ Then subdivide $[\tau_\mu, 1-\tau_\mu]$ into $N/2$ mesh intervals, and subdivide each of the other four subintervals into $N/8$ mesh intervals. For more details on S~--~mesh, refer \citet{madden2003uniformly}.

\vspace{1em}

\textit{Bakhvalov - Shishkin mesh (BS~--~mesh $\Omega_{BS}$)}:
We consider a modification of the Shishkin mesh that integrates an idea from \citet{bakhvalov1969optimization}, wherein mesh condenses within boundary layers by inverting the associated boundary layer terms. Here, we choose transition points similar to those of the S~--~mesh. The BS~--~mesh corresponding to the case $\varepsilon = \mu$ is detailed in \cite{linss2009layer}. We now propose a BS~--~type mesh suitable for the more general case $0 < \varepsilon \leqslant \mu \leqslant 1.$  We now make the very mild assumption that $\tau_\mu = ({\sigma \mu} / {\lambda}) \ln N \quad \text{and} \quad \tau_\varepsilon = ({\sigma \varepsilon} / {\lambda})  \ln N,$ as otherwise $ N^{-1} $ is exponentially smaller in magnitude than the parameters $\varepsilon$ and $\mu.$ In this case, we assume that $\varepsilon \leqslant \mu \leqslant N^{-1}$, which is typical in practice. The interval $[\tau_\mu, 1-\tau_\mu]$ is uniformly dissected into $N/2$ subintervals. The interval $[0, \tau_\varepsilon]$ is partitioned into $N/8$ mesh intervals by inverting the function $\exp\left( - \lambda x / (2 \varepsilon )\right).$ We specify $x_i$, for $i=0,1, \ldots, N/8,$ so that $\exp\left( - \lambda x_i / (2 \varepsilon )\right)$ is a linear function in $i$. i.e., we set 
\begin{align*}
	\exp\left( - \lambda x_i / (2 \varepsilon )\right) = \mathfrak{R}i + \mathfrak{S}
\end{align*}
and choose the unknowns $\mathfrak{R}$ and $\mathfrak{S}$ so that $x_0 = 0$ and $x_{N/8} = \tau_\varepsilon$. Similarly, the intervals $[\tau_\varepsilon, \tau_\mu], [1-\tau_\mu, 1-\tau_\varepsilon],$ and $[1-\tau_\varepsilon,1]$ are partitioned into $N/8$ mesh intervals each by inverting the functions $\exp \left( - \lambda x / (2 \mu) \right), ~ \exp \left( - \lambda (1-x) / (2 \mu) \right)$ and $\exp \left( - \lambda (1-x) / (2 \varepsilon) \right)$ respectively. This gives
\begin{align*}  \label{BS mesh points}
	x_i = \begin{cases}
		\dfrac{-2 \varepsilon}{\lambda} \ln \left( 1 - \dfrac{8i}{N} \left( 1 - N^{-\sigma/2} \right) \right), \quad i=0, \ldots, N/8    \\[8pt] 
		\dfrac{-2 \mu}{\lambda} \ln \left( \dfrac{8i}{N} \left( N^{-\sigma/2}-N^{-\sigma \varepsilon / (2 \mu)} \right) + \left( 2 N^{-\sigma \varepsilon / (2 \mu)} - N^{-\sigma/2} \right) \right), \quad i = N/8+1, \ldots, N/4-1   \\[8pt]
		\tau_\mu + \left( \dfrac{1-2 \tau_\mu}{N/2} \right) \left( i - \dfrac{N}{4} \right), \quad i = N/4, \ldots, 3N/4   \\[8pt]
		1 + \dfrac{2 \mu}{\lambda} \ln \left( \dfrac{8i}{N} \left( N^{-\sigma \varepsilon / (2 \mu)} - N^{-\sigma/2}  \right) + \left( 7 N^{-\sigma/2} - 6 N^{-\sigma \varepsilon / (2 \mu)}  \right) \right), \quad i=3N/4+1, \ldots, 7N/8-1 \\[8pt]
		1 + \dfrac{2 \varepsilon}{\lambda} \ln \left( 1 - \dfrac{8}{N} \left( 1 - N^{-\sigma/2}  \right) \left( N - i \right)\right), \quad i = 7N/8, \ldots, N.
	\end{cases}
\end{align*}


\begin{lemma}[Discrete maximum principle]   \label{L: discrete maximum principle}
	If $ {R}^{(0),{\,}N} \vec{Y}(x_0) \geqslant 0, ~ {R}^{(1),{\,}N} \vec{Y}(x_N) \geqslant 0 $ and $L^N \vec{Y}(x_i) \geqslant 0$ for $1 \leqslant i \leqslant N-1,$  then $\vec{Y}(x_i) \geqslant 0$ for all $0 \leqslant i \leqslant N.$
\end{lemma}

\begin{remark}
	If the stiffness matrix associated with the discrete operator $L^N$, defined by the scheme \eqref{Eq: left boundary conditions - cubic spline}~--~\eqref{Eq: right boundary conditions - cubic spline}, is an $M$-matrix, then the discrete maximum principle is satisfied. The next lemma confirms that the discretization of $L^N$ indeed yields an $M$-matrix.
\end{remark}

\begin{lemma} \label{L: m matrix}
	Let $N_0$ and $N_1$ be sufficiently large positive integers such that

	\begin{equation}
		32 \lambda^* \lambda^{-2} \sigma^2 N^{-2} \ln^2 N 
		< 3 \qquad \text{for all } N \geqslant N_0 \quad \text{(S~--~mesh)},
		\label{Eq: assumption for m-matrix proof on shishkin mesh}
	\end{equation}
	\begin{equation}
		2 \lambda^* \lambda^{-2} \ln^2 \left( 1 + 8 N^{-1} ( N^{-\sigma / 2} - 1 ) \right)
		< 3 \qquad \text{for all } N \geqslant N_1 \quad \text{(BS~--~mesh)}.
		\label{Eq: assumption for m-matrix proof on bs mesh}
	\end{equation}
	
	\noindent Then, for all $i=0,\ldots,N,$ the coefficients of the discretized system satisfy:
	\begin{align*}
		A_{1,i}^- < 0, \quad A_{1,i}^+ < 0, \quad A_{1,i}^c > 0, \quad |A_{1,i}^c| > |A_{1,i}^-|+|A_{1,i}^+| , \\
		A_{2,i}^- < 0, \quad A_{2,i}^+ < 0, \quad A_{2,i}^c > 0, \quad |A_{2,i}^c| > |A_{2,i}^-|+|A_{2,i}^+|.
	\end{align*}
	As a consequence, the stiffness matrix arising from the numerical scheme \eqref{Eq: left boundary conditions - cubic spline}~--~\eqref{Eq: right boundary conditions - cubic spline}, applied to the system \eqref{P: system general}~--~\eqref{P: robin boundary conditions}, satisfies the discrete maximum principle. Moreover, the scheme is uniformly stable with respect to the perturbation parameters in the maximum norm.
\end{lemma}

\begin{proof}
	~ From \eqref{Eq: left boundary conditions - cubic spline}~--~\eqref{Eq: right boundary conditions - cubic spline}, it is straightforward to verify that $A_{1,i}^c> 0 $ for all $i$.
	Using the assumption in \eqref{Eq: assumption for m-matrix proof on shishkin mesh}, we first consider the S-mesh, for which $h_1 = 8 \sigma \varepsilon \lambda^{-1} N^{-1} \ln N.$ Then, the coefficient $A_{1,0}^+$ satisfies
	\begin{equation*}
		A_{1,0}^+ = \dfrac{-3 \varepsilon^2 \beta_1}{h_1^2} + \dfrac{c_{11}(x_1) \beta_1}{2}
		< \dfrac{\beta_1 (-3 + 32 \lambda^* \lambda^{-2} \sigma^2 N^{-2} \ln^2 N)}{64 \lambda^{-2} \sigma^2 N^{-2} \ln^2 N } < 0
	\end{equation*}
	for all $N \geqslant N_0.$ In the case of BS mesh, where $h_1 = -2 \varepsilon \lambda^{-1} \ln \left( 1 + 8 N^{-1} (N^{-\sigma/2} - 1) \right)$, a similar argument using \eqref{Eq: assumption for m-matrix proof on bs mesh} yields
	\begin{equation*}
		A_{1,0}^+ 
		< \frac{\beta_1 \left( -3 + 2 \lambda^* \lambda^{-2} \ln^2 \left( 1 + 8 N^{-1} (N^{-\sigma/2} - 1) \right) \right)}{4 \lambda^{-2} \ln^2 \left( 1 + 8 N^{-1} (N^{-\sigma/2} - 1) \right)} < 0,
	\end{equation*}
	for all \( N \geqslant N_1 \). Furthermore, the difference \( |A_{1,0}^c| - |A_{1,0}^+| \) satisfies
	\begin{align*}
		|A_{1,0}^c| - |A_{1,0}^+| 
		&= \frac{3 \varepsilon}{h_1} \left( \alpha_1 + \frac{\varepsilon \beta_1}{h_1} \right) + c_{11}(x_0) \beta_1 
		- \frac{3 \varepsilon^2 \beta_1}{h_1^2} + \frac{c_{11}(x_1) \beta_1}{2} \\
		&> \frac{3 \varepsilon \alpha_1}{h_1} + \frac{3 \lambda^* \beta_1}{2} > 0.
	\end{align*}
	A similar argument can be applied to show that $A_{1,N}^- < 0$, $A_{1,N}^c > 0$ and \mbox{$|A_{1,N}^c|-|A_{1,N}^-|>0 $.} From the discretization scheme in \eqref{Eq: discretization central scheme}, it follows that for  $i = 1, \ldots, N-1,$ the coefficients satisfy
	\begin{equation*}
		A_{1,i}^- < 0, \quad A_{1,i}^+ < 0, \quad A_{1,i}^c > 0, \quad |A_{1,i}^c|-|A_{1,i}^-|-|A_{1,i}^+| > 0,
	\end{equation*}
	on both the Shishkin and BS meshes. By a similar argument, we can show that for all $i$,
	\begin{equation*}
		A_{2,i}^- < 0, \quad A_{2,i}^+ < 0, \quad A_{2,i}^c > 0, \quad |A_{2,i}^c|-|A_{2,i}^-|-|A_{2,i}^+| > 0,
	\end{equation*}
	under both mesh types. Hence, the discrete operator \eqref{P: discretization system general}~--~\eqref{P: discretization robin boundary conditions} is parameter-uniform stable.
	\hfill
\end{proof}


\begin{lemma}[Discrete stability result]   \label{L: discrete stability result}
	If $\vec{Y}$ is any mesh function satisfying \eqref{P: discretization system general}~--~\eqref{P: discretization robin boundary conditions}, then
	\[
	\| \vec{Y}(x_i) \| \leqslant C \max \left\{
	\|  {  R}^{(0),{\,}N} \vec{Y}(0)  \|, ~ \|  {  R}^{(1),{\,}N} \vec{Y}(1)  \|, ~  \dfrac{1}{\lambda^2} \|   L^N \vec{Y}  \| 
	\right\}, ~ 0 \le i \le N
	\]	
\end{lemma}
\begin{proof}
	Define
	\begin{equation*}
		\vec{\psi}^\pm(x_i) = \max \left \{  \| {  R}^{(0),N} \vec{Y}(0)  \|, ~ \| {  R}^{(1),N} \vec{Y}(1) \|, ~	\dfrac{1}{\gamma^2}  \| L \vec{Y} {\,}  \| \right\} \begin{pmatrix} 1 \\ 1
		\end{pmatrix} \pm \vec{Y}(x_i), ~ \text{on} ~ \Bar{\Omega}^N.
	\end{equation*}
	It is easy to verify that ${  R}^{(0),{\,}N} \vec{\psi}^\pm(0) \geqslant 0, ~ 
	{  R}^{(1),{\,}N} \vec{\psi}^\pm(1) \geqslant 0$ and $L^N \psi^\pm(x_i) \geqslant 0$ on $\Omega^N.$ Using Lemma~ \ref{L: maximum principle}, it easily follows that $\vec{\psi}^\pm(x_i) \geqslant \vec{0}$ on $\Bar{\Omega}^N$ and the required result follows.
	\hfill
\end{proof}


The solutions $\vec{Y}$ of the discrete problem are decomposed in a similar manner to the decomposition of the solution $
\vec{y}.$ Thus,
\begin{equation*}   \label{Eq: solution decomposition discretized}
	\vec{Y} = \vec{V} + \vec{W},
\end{equation*}
where $\vec{V}=(V_1,V_2)^T$ is the solution of the inhomogeneous problem
\begin{equation}   \label{Eq: discretized solution decomposition smooth part}
	L^N \vec{V} = \vec{f} ~  \text{on} ~ \Omega^N, \quad
	{R}^{(0),{\,}N} \vec{V}(0) = B(0)^{-1} \vec{f}(0), \quad {  R}^{(1),{\,}N} \vec{V}(1) = B(1)^{-1} \vec{f}(1),
\end{equation}
and $\vec{W}=(W_1,W_2)^T$ is the solution of the homogeneous problem 
\begin{equation}   \label{Eq: discretized solution decomposition layer part}
	{L}^N \vec{W} = \vec{0} ~ \text{on} ~ \Omega^N, \quad {R}^{(0),{\,}N} \vec{W}(0) =   P -{R}^{(0),{\,}N} \vec{V}(0), \quad {R}^{(1),{\,}N} \vec{W}(1) =   Q -{R}^{(1),{\,}N} \vec{V}(1).
\end{equation}
The next section deals with the error estimates related to the discretized smooth and layer components.

\section{Error analysis} \label{S: error analysis}
In this section, we examine the truncation error and the stability of the proposed numerical scheme. These results form the basis for proving convergence. We conclude the section by stating the main result on parameter-uniform convergence. For $i = 1,\dotsc, N-1$, the truncation errors are given by
\begin{align}   
	{T}_{y_1,i} &= - \varepsilon^2 \left( \delta^2 - \dfrac{d^2}{dx^2} \right) y_1(x_i),  \label{truncation error central scheme y1}  \\
	{T}_{y_2,i} &= - \mu^2 \left( \delta^2 - \dfrac{d^2}{dx^2} \right) y_2(x_i).  \notag
\end{align}
To analyse the truncation error for $x_0 =0,$ we see that
\begin{equation*}  \label{truncation error x0}
	{T}_{y_1,0} = A_{1,0}^c y_{1,0} + A_{1,0}^+ y_{1,1} + B_{1,0}^c y_{2,0} +  B_{1,0}^+ y_{2,1} - F_{1,0}^- - F_{1,0}^c f_{1,0} - F_{1,0}^+ f_{1,1}.
\end{equation*} 
Expanding $y_1$ at $x_0$ using Taylor series and the system \eqref{P: system general}~--~\eqref{P: robin boundary conditions}, the truncation error takes the form
\begin{align*}   \label{truncation error x0 taylor series expansion}
	{T}_{y_1,0} = T_{0,0}y_1(x_0) + T_{1,0}y_1'(x_0) +  T_{2,0}y_1''(x_0) + T_{3,0}y_1'''(x_0) + T_{4,0}y_1^{(4)}(\xi),
\end{align*}
where $\xi \in (0,1)$ and
\begin{align*}
	T_{0,0} &= A_{1,0}^c + A_{1,0}^+ - \dfrac{3 \varepsilon \alpha_1}{h_1} - F_{1,0}^c c_{11}(x_0)  - F_{1,0}^+ c_{11}(x_1), \\
	T_{1,0} &= h_1 A_{1,0}^+ +  \dfrac{3 \varepsilon^2 \beta_1}{h_1} - F_{1,0}^+  c_{11}(x_1) h_1,  \\
	T_{2,0} &= \dfrac{h_1^2 A_{1,0}^+}{2} + \varepsilon^2 \left( F_{1,0}^c + F_{1,0}^+ \right) - \dfrac{h_1^2 F_{1,0}^+ c_{11}(x_1)}{2},  \\
	T_{3,0} &= \dfrac{h_1^3 A_{1,0}^+}{3!} + \varepsilon^2 h_1 F_{1,0}^+ - \dfrac{F_{1,0}^+ c_{11}(x_1) h_1^3}{3!}, \\
	T_{4,0} &= \dfrac{h_1^4 A_{1,0}^+}{4!} + \dfrac{\varepsilon^2 h_1^2 F_{1,0}^+}{2}  - \dfrac{F_{1,0}^+ c_{11}(x_1) h_1^4}{4!}.
\end{align*}

\noindent From these expressions, we obtain the following conditions: 
\begin{align*}
	T_{0,0} = 0, \quad T_{1,0} = 0, \quad
	T_{2,0} = 0, \quad T_{3,0} = 0, \quad
	T_{4,0} = \dfrac{1}{8} \varepsilon^2 \beta_1 h_1^2.
\end{align*}
Therefore, the truncation error for $y_1$ at $x_0$ is
\begin{equation}   \label{Eq: y1 at 0 truncation}
	|  {T}_{y_1,0} |  \leqslant C \varepsilon^2 \beta_1 h_1^2 | y_{1,0}^{(4)} |_{(x_0, x_1)}.
\end{equation}
Similarly, the truncation error at $x_N$ for $y_1$ satisfies
\begin{equation}  \label{Eq: y1 at N truncation}
	|  {T}_{y_1,N} |   \leqslant C \varepsilon^2 \delta_1 h_N^2 | y_{1,N}^{(4)} |_{(x_{N-1}, x_N)}.
\end{equation}
Following the same approach used in the analysis of $y_1$, we obtain
\begin{align*}  \label{Eq: y2 at 0 and N truncation}
	|  {T}_{y_2,0} |  \leqslant C \mu^2 \beta_2 h_1^2 | y_{2,0}^{(4)} |_{(x_0, x_1)}, \qquad
	|  {T}_{y_2,N} |  \leqslant C \mu^2 \delta_2 h_N^2 | y_{2,N}^{(4)} |_{(x_{N-1}, x_N)}. 
\end{align*}

	
\noindent The following lemma gives some estimates of the mesh sizes that will be used later.
\begin{lemma}     \label{L: mesh size of BS mesh}
	The step sizes of the Bakhvalov--Shishkin mesh $\Omega_{BS}^N$ satisfy
	\begin{equation*}
		h_i \leqslant C N^{-1} \quad \text{for all  }  i=1,2,\dotsc N.
	\end{equation*}
\end{lemma}

\begin{proof}
	 We estimate the mesh widths $h_i = x_i - x_{i-1} $ for various ranges of $i.$ We adapt the argument from \citet{linss1999upwind} to the present setting.
	
	\noindent \textit{Case 1.} Let $i=1,\dotsc,N/8.$ In this region, the mesh widths satisfy
	\begin{equation}  \label{Eq: hi for case 1}
		h_i \leqslant \dfrac{2 \varepsilon}{\lambda} \ln \left( \dfrac{i-1}{i} \right).
	\end{equation}
	To estimate this expression, we use the inequality $i \exp(2/i) \geqslant i \left( 1 + 2/i \right) \geqslant i-1,$ which implies
	\begin{equation*}
		\ln \left( \dfrac{i-1}{i} \right) \leqslant \dfrac{2}{i}.
	\end{equation*}
	Substituting this bound into \eqref{Eq: hi for case 1}, we obtain $h_i \leqslant 4 \varepsilon / \lambda i.$ Since $\varepsilon \leqslant N^{-1}$ and $i \geqslant 1,$ it follows that $h_i \leqslant C N^{-1}$ in this region.
	
	\noindent \textit{Case 2.} For $i = N/8 + 1, \dotsc, N/4-1,$ 
	\begin{equation*}  \label{Eq: hi for case 2}
		h_i \leqslant \dfrac{2 \mu}{\lambda} \ln \left( \dfrac{i-1}{i} \right) \leqslant \dfrac{4 \mu}{\lambda i} \leqslant C N^{-1},
	\end{equation*}
	using $\varepsilon \leqslant \mu \leqslant N^{-1}. $
	
	\noindent \textit{Case 3.} For $i = N/4, \dotsc, 3N/4,$ the step size is uniform:
	\begin{equation*}
		h_i = \dfrac{1-2 \tau_\mu}{N/2} \leqslant 2 N^{-1}.
	\end{equation*}
	
	\noindent \textit{Case 4.} Consider the interval $i = 3N/4 + 1, \dotsc, 7N/8-1.$ Then, the step sizes satisfy
	\begin{equation*}  \label{Eq: hi for case 4}
		h_i \leqslant \dfrac{2 \mu}{\lambda} \ln \left( \dfrac{i}{i-1} \right).
	\end{equation*}
	Since the inequality $(i - 1)\exp(2/i) \geqslant i$ implies $\ln\left( \frac{i}{i - 1} \right) \leqslant \frac{2}{i}$, it follows that
	\begin{equation*}
		h_i \leqslant \dfrac{4 \mu}{\lambda i} \leqslant C N^{-1}, \quad \text{using } \varepsilon \leqslant \mu \leqslant N^{-1}.
	\end{equation*}
	
	\noindent \textit{Case 5.} For $i = 7N/8, \dotsc, N,$
	\begin{equation*}
		h_i \leqslant \dfrac{4 \varepsilon}{\lambda i} \leqslant C N^{-1}.
	\end{equation*}
	In all cases, we conclude that $h_i \leqslant C N^{-1}$  and the proof is complete.
	\hfill
\end{proof}


\noindent The following results give error estimates on the regular and layer components separately, on both S-mesh and BS-mesh.
\begin{lemma} \label{L: truncation error for v}
	The smooth components $\vec{v}$ and $\vec{V}$ from \eqref{Eq: solution decomposition smooth part} and \eqref{Eq: discretized solution decomposition smooth part} respectively satisfy the following error bound on S~--~mesh or BS~--~mesh:
	\begin{equation*}
		\|L^N(\vec{V} - \vec{v}) \|_{\Omega^N} \leqslant C N^{-2},
	\end{equation*}
	where $C$ is a generic constant.
\end{lemma}
\begin{proof}
	~The proof relies on \citet{miller2012fitted} and the bounds in \eqref{Eq: y1 at 0 truncation}~--~\eqref{Eq: y1 at N truncation}. For both mesh types, the step size satisfies $h_i \leqslant CN^{-1}.$ Since $\varepsilon \leqslant \mu \leqslant N^{-1}$, for each $0 \leqslant i \leqslant N,$
	\begin{align*}
		|L_1^N (\vec{V} - \vec{v})(x_i)| \leqslant C \varepsilon^2 h_i^2 \| v_1^{(4)} \| \leqslant C \varepsilon^2 h^2 (1 + \varepsilon^{-2}) \leqslant C N^{-2}
	\end{align*}
	The bound on $\|v_1^{(4)}\|$ is given in \eqref{Eq: v1 v2 k=0 to 4}. A similar argument applies to $\|L_2^N (\vec{V} - \vec{v})\|$ and the required result follows.
	\hfill	
\end{proof}


\begin{lemma}     \label{L: truncation error for w at i = 0, N}
	The layer components $\vec{w}$ and $\vec{W}$ from \eqref{Eq: solution decomposition layer part} and \eqref{Eq: discretized solution decomposition layer part}, respectively satisfy the following error bounds at the endpoints $i =0$ and $N$:
	\begin{equation*}
		| L^N (\vec{W} - \vec{w}) (x_i) | \leqslant 
		\begin{cases}
			C N^{-2} \ln^2 N \quad &\text{on S~--~mesh}\\
			C N^{-2} \quad &\text{on BS~--~mesh}.
		\end{cases}
	\end{equation*}
\end{lemma}
\begin{proof}
	~ On a Shishkin mesh, the mesh size near the boundary satisfies $h = 8 \sigma \varepsilon \lambda^{-1} N^{-1} \ln N \leqslant C N^{-1} \ln N.$ The error estimate for the layer component proceeds as follows. From \eqref{Eq: w1 w2 k=3,4}, we obtain the bounds for $\|  w_1^{(4)}\|,$ leading to
	\begin{align*}
		| L_1^N (\vec{W} - \vec{w}) (x_i) | & \leqslant C \varepsilon^2 h^2 \|w_1^{(4)}\| \leqslant C \varepsilon^2 h^2 (\varepsilon^{-4} \mathbb{B}_\varepsilon + \varepsilon^{-2} \mu^{-2} \mathbb{B}_\mu ) \\
		& \leqslant Ch^2 \varepsilon^{-2} \leqslant C N^{-2} \ln^2 N.
	\end{align*}
	On BS mesh, however, we have the improved mesh size bound $h \leqslant C N^{-1}, $ as established in Lemma~\ref{L: mesh size of BS mesh} and the same argument gives
	\begin{align*}
		| L_1^N (\vec{W} - \vec{w}) (x_i) | & \leqslant C \varepsilon^2 h^2 \|w_1^{(4)}\|  \leqslant C N^{-2} .
	\end{align*}
	A similar argument applies to $| L_2^N (\vec{W} - \vec{w}) (x_i) |$ and hence the required result follows.
	\hfill
\end{proof}


\begin{lemma}  \label{L: truncation error for w at [tau_mu, 1-tau_mu]}
	In the outer region $[\tau_\mu, 1-\tau_\mu]$, the truncation error satisfies
	\begin{equation*}
		| L^N (\vec{W} - \vec{w})(x_i) | \leqslant C N^{-2}, 
	\end{equation*}
	for $N/4 \leqslant i \leqslant 3N/4,$ on both S-mesh and BS-mesh.
\end{lemma}
\begin{proof}
	~	In this region, the fitted mesh is uniform for both S-mesh and BS-mesh, with $h_i \leqslant C N^{-1}.$
	From \eqref{truncation error central scheme y1}, it follows that for $N/4 \leqslant i \leqslant 3N/4,$
	\begin{equation*}
		|L_1^N (\vec{W} - \vec{w})(x_i)| = \varepsilon^2 \left | \left( \delta^2 - \dfrac{d^2}{dx^2} \right) w_1(x_i) \right |.
	\end{equation*}
	For the layer region $[\tau_\mu, 1/2],$ applying the derivative bounds from \eqref{Eq: derivative bounds for w1 w2 for k=0 to 2} yields
	\begin{align*}
		\varepsilon^2 \left | \left( \delta^2 - \dfrac{d^2}{dx^2} \right) w_1(x_i) \right | & \leqslant C \varepsilon^2 |w_1''| _{[x_{i-1},x_{i+1}]} \leqslant C (\mathbb{B}_\varepsilon + \varepsilon^2 \mu^{-2} \mathbb{B}_\mu)  \\
		& \leqslant \| \mathbb{B}_\mu \|_{[x_{i-1},x_{i+1}]} \leqslant \mathbb{B}_\mu(x_{i-1}) \leqslant C N^{-2}.
	\end{align*}
	Similarly, in the interval $[1/2, 1-\tau_\mu],$ we have
	\begin{equation*}
		\varepsilon^2 \left | \left( \delta^2 - \dfrac{d^2}{dx^2} \right) w_1(x_i) \right | \leqslant C N^{-2}.
	\end{equation*}
	This establishes the desired estimate.
	\hfill
\end{proof}


\begin{lemma}  \label{L: truncation error in the inner region 1}
	In the boundary layers $(0, \tau_\varepsilon)$ and $(1-\tau_\varepsilon,1)$, the discrete layer component $\vec{W}$ satisfies
	\begin{equation*}
		| L^N(\vec{W} - \vec{w}) (x_i) | \leqslant 
		\begin{cases}
			C N^{-2} \ln^2 N \quad &\text{on S-mesh}\\
			C N^{-2} \quad &\text{on BS-mesh}.
		\end{cases}
	\end{equation*}
	for $ 0 < i < N/8$ and $7N/8 < i < N$.
\end{lemma}
\begin{proof}
	~ The result follows by a similar argument as in Lemma \ref{L: truncation error for w at i = 0, N}.
	\hfill
\end{proof}


\begin{lemma}  \label{L: truncation error in the inner region 2} 
	The layer component $\vec{W}$ satisfies the error bound 
	\begin{equation*}
		| L^N(\vec{W} - \vec{w}) (x_i) | \leqslant 
		\begin{cases}
			C N^{-2} \ln^2 N \quad &\text{on S-mesh}\\
			C N^{-2} \quad &\text{on BS-mesh}.
		\end{cases}
	\end{equation*}
	in the regions $[\tau_\varepsilon, \tau_\mu)$ and $(1-\tau_\mu, 1-\tau_\varepsilon];$ that is, for $N/8 \leqslant i < N/4$ and $3N/4 < i \leqslant 7N/8, $ where $C$ is a generic constant.
\end{lemma}

\begin{proof}
	~	We decompose $\vec{w} = \vec{w_{\varepsilon}} + \vec{w_{\mu}}$ to analyze the layer error. A direct estimate gives
	\begin{equation*}
		| L_1^N (\vec{W_{\varepsilon}} - \vec{w_{\varepsilon}})(x_i) | \leqslant C \varepsilon^2 \| w_{1,\varepsilon}'' \|_{[x_{i-1},x_{i+1}]}
		\leqslant C \| \mathbb{B}_\varepsilon \|_{[x_{i-1},x_{i+1}]}.
	\end{equation*}
	For $x_{i-1} \leqslant \tau_\varepsilon,$ applying the bound from lemma~\ref{L: w derivatives} yields
	\begin{equation*}
		| L_1^N (\vec{W_{\varepsilon}} - \vec{w_{\varepsilon}})(x_i) | \leqslant C \mathbb{B}_\varepsilon(x_{i-1}) \leqslant C N^{-2}.
	\end{equation*}
	Similarly, using lemma~\ref{L: w derivatives} for $w_{1,\mu},$ we obtain
	\begin{align*}
		| L_1^N (\vec{W_{\mu}} - \vec{w_{\mu}}) | & \leqslant C \varepsilon^2 h^2 \| w_{1,\mu}^{(4)} \| \leqslant C \varepsilon^2 h^2 \varepsilon^{-2} \mu^{-2} \mathbb{B}_\mu \\
		& \leqslant C N^{-2} \ln^2 N,
	\end{align*}
	in case of Shishkin mesh. In case of BS-mesh,
	\begin{equation*}
		| L_1^N (\vec{W_{\mu}} - \vec{w_{\mu}}) |  \leqslant C N^{-2} .
	\end{equation*}
	Combining both parts, we have
	\begin{equation*}
		| L_1^N (\vec{W} - \vec{w})(x_i) | \leqslant 	
		\begin{cases}
			C N^{-2} \ln^2 N \quad &\text{on S-mesh}\\
			C N^{-2} \quad &\text{on BS-mesh}.
		\end{cases}
	\end{equation*}
	The result then follows similarly for  $L_2^N.$
	\hfill
\end{proof}


\begin{theorem}  \label{Th: main theorem}
	Consider the exact solution $\vec{y}(x) = (y_1(x), y_2(x))^T$ of the system \eqref{P: system general}~--~\eqref{P: robin boundary conditions}, and let $\vec{Y}=(Y_1, Y_2)^T$ denote the numerical approximation obtained via the scheme \eqref{P: discretization system general}~--~\eqref{P: discretization robin boundary conditions}. If the conditions of Lemma~\ref{L: m matrix} are satisfied and $\varepsilon \leqslant \mu \leqslant N^{-1},$ then the following parameter-uniform error bound holds:
	\begin{equation*}
		| \vec{y}(x_i) - \vec{Y}(x_i) | \leqslant
		\begin{cases}
			C N^{-2} \ln^2 N \quad &\text{on S~--~mesh} \\
			C N^{-2} \quad &\text{on BS~--~mesh}
		\end{cases}
		\qquad	i = 0, \dotsc, N.
	\end{equation*}
\end{theorem}

\begin{proof}
	~ For the case of S-mesh, we define a barrier function $\vec{\Phi}_S = (\Phi_1, \Phi_2)^T$ by
	\begin{align*}
		\vec{\Phi}_S(x_i) = C \left (N^{-2} \ln^2 N  + N^{-2} \dfrac{\tau_\varepsilon^2}{\varepsilon} \Psi_1(x_i)  + N^{-2} \dfrac{\tau_\mu^2}{\mu} \Psi_2(x_i) \right) 
		\begin{pmatrix}
			1 \\ 1
		\end{pmatrix},
		\quad i = 1, \dotsc, N-1.
	\end{align*}
	Here, the functions $\Psi_1$ and $\Psi_2$ are defined piecewise as:
	\begin{equation*}
		\Psi_1(x) =
		\begin{cases}
			x / \tau_\varepsilon,   \quad & 0 \leqslant x \leqslant \tau_\varepsilon \\
			1,  \quad &\tau_\varepsilon \leqslant x \leqslant 1-\tau_\varepsilon \\
			(1-x) / \tau_\varepsilon, \quad &1-\tau_\varepsilon \leqslant x \leqslant 1
		\end{cases}
	\end{equation*}
	and
	\begin{equation*}
		\Psi_2(x) =
		\begin{cases}
			x / \tau_\mu,   \quad & 0 \leqslant x \leqslant \tau_\mu \\
			1,  \quad &\tau_\mu \leqslant x \leqslant 1-\tau_\mu \\
			(1-x) / \tau_\mu, \quad &1-\tau_\mu \leqslant x \leqslant 1.
		\end{cases}
	\end{equation*}
	With this construction, it can be shown after simplification that
	\begin{equation*}
		\begin{aligned}
			L_1^N \vec{\Psi} \geqslant\; & C N^{-2} \ln^2 N \left( A_{1,i}^- + A_{1,i}^c + A_{1,i}^+ + B_{1,i}^- + B_{1,i}^c + B_{1,i}^+ \right) \\
			& + C \frac{\tau_\varepsilon^2}{\varepsilon} N^{-2} \left[
			(A_{1,i}^- + B_{1,i}^-) \Psi_{1,i-1} +
			(A_{1,i}^c + B_{1,i}^c) \Psi_{1,i} +
			(A_{1,i}^+ + B_{1,i}^+) \Psi_{1,i+1}
			\right] \\
			& + C \frac{\tau_\mu^2}{\mu} N^{-2} \left[
			(A_{2,i}^- + B_{2,i}^-) \Psi_{2,i-1} +
			(A_{2,i}^c + B_{2,i}^c) \Psi_{2,i} +
			(A_{2,i}^+ + B_{2,i}^+) \Psi_{2,i+1}
			\right].
		\end{aligned}
	\end{equation*}
	From Lemmas~\ref{L: truncation error for v}-\ref{L: truncation error in the inner region 2}, it follows that the truncation error is bounded by
	\begin{equation*}
		|T_{y_1,i}| = |L_1^N (\vec{Y} - \vec{y})| \leqslant |L_1^N \vec{\Psi}|.
	\end{equation*}
	Similarly, for the second component of the system,
	\begin{equation*}
		|T_{y_2,i}| = |L_2^N (\vec{Y} - \vec{y})| \leqslant |L_2^N \vec{\Psi}|.
	\end{equation*}
	By Lemma~\ref{L: m matrix}, the discrete operator $L^N$ corresponds to an M-matrix, implying that its inverse is uniformly bounded with respect to the parameters. Therefore, we obtain the error estimate
	\begin{equation*}
		\| \vec{Y} - \vec{y} \| \leqslant \| \vec{\Psi}\| \leqslant C N^{-2} \ln^2 N
	\end{equation*}
	in case of Shishkin mesh. For BS mesh, we define a similar barrier function:
	\begin{align*}
		\vec{\Phi}_{BS}(x_i) = C \left (N^{-2}  + N^{-2} \dfrac{\tau_\varepsilon^2}{\varepsilon} \Psi_1(x_i)  + N^{-2} \dfrac{\tau_\mu^2}{\mu} \Psi_2(x_i) \right) 
		\begin{pmatrix}
			1 \\ 1
		\end{pmatrix},
		\quad i = 1, \dotsc, N-1,
	\end{align*}
	Following analogous steps, we conclude that
	\begin{equation*}
		\| \vec{Y} - \vec{y} \| \leqslant \| \vec{\Psi}\| \leqslant C N^{-2} .
	\end{equation*}
	This completes the proof.
	\hfill
\end{proof}  %

	The next result demonstrates that a parameter-uniform global approximation is obtained by applying piecewise linear interpolation to the computed solution $\vec{Y}$.
\begin{theorem}
	The numerical solution $\vec{Y}$ from \eqref{Eq: left boundary conditions - cubic spline}~--~\eqref{Eq: right boundary conditions - cubic spline} converges uniformly to the exact solution $\vec{y}$ of \eqref{P: system general}~--~\eqref{P: robin boundary conditions}, satisfying the following parameter-uniform error estimate.
	\begin{equation*}
		\| \vec{Y}_h - \vec{y} \| \leqslant \|
		\begin{cases}
			C N^{-2} \ln^2 N \quad &\text{on S~--~mesh} \\
			C N^{-2} \quad &\text{on BS~--~mesh},                                     
		\end{cases}
	\end{equation*}
	where $\vec{Y}_h$ is the piecewise linear interpolant of $\vec{Y}$ on $\Bar{\Omega}.$
\end{theorem}

\begin{proof}
	~ Consider the error using the triangle inequality,
	\begin{equation*}
		\| \vec{Y}_h - \vec{y} \| \leqslant \| \vec{Y}_h - \vec{y}_h \| + \| \vec{y}_h - \vec{y} \|.
	\end{equation*}
	Here, $\vec{y}_h$ represents the piecewise linear interpolant of $\vec{y}$ at the mesh nodes. Applying standard estimates for linear interpolation, we arrive at the desired result.
	\hfill
\end{proof}


\section{Numerical results} \label{S: numerical results}
We consider the following examples from \citet{matthews2000parameter}, \citet{basha2015uniformly} and \citet{kaushik2024adaptive}. 


\begin{example} \label{E: example 1}
	\begin{align*}
		-\varepsilon^2 y_1'' + (x+1)^2y_1 - (x+0.5) y_2 & = x^5 - 0.08 \\
		-\mu^2 y_2'' - y_1 + 2 y_2 &= \sin(\pi x), 
	\end{align*}
	with the Robin boundary conditions
	\begin{alignat*}{2}
		y_1(0) - {\varepsilon} y_1'(0) &= 1, \qquad
		y_1(1) + {\varepsilon} y_1'(1) &&= 1, \\
		y_2(0) - {\mu} y_2'(0) &= 1, \qquad
		y_2(1) + {\mu} y_2'(1) &&= 1.
	\end{alignat*}
\end{example}

	Consider the numerical solution $\vec{Y}^N$ obtained from the discrete system \eqref{P: discretization system general}~--~\eqref{P: discretization robin boundary conditions} on a non-uniform mesh $\Omega^N$ consisting of $N$ subintervals. This mesh may be of either S-type ($\Omega_S^N$) or BS-type ($\Omega_{BS}^N$). The corresponding maximum pointwise error is given by $\| \vec{Y}^N - \vec{y} \|_{\Omega^N}$, where $\vec{y}$ is the exact solution. However, due to the unavailability of an explicit solution for Example~\ref{E: example 1}, the error is estimated numerically. To estimate this error, we construct a finer mesh $\Omega^{5N}$ of the same type (S or BS) as $\Omega^N$, using the same transition points but with five times as many intervals. Let $\vec{Y}^{5N}$ denote the numerical solution computed on this finer mesh. Then, the quantity 
\begin{equation*}
	E_{\varepsilon, \mu}^N = \| \vec{Y}^N - \vec{Y}^{5N} \|_{\Omega^N}
\end{equation*}
is used as an approximation to the maximum pointwise error.

	Given $\varepsilon = 10^{-j}$ for some non-negative integer $j$, the quantity $E_{\varepsilon}^N$ is defined as the maximum of the errors over multiple values of $\mu$, that is, 
\begin{equation}  \label{Eq: E eps N}
	E_{\varepsilon}^N = \max \{ E_{\varepsilon, 10^{-3}}^N, E_{\varepsilon,10^{-4}}^N, \ldots, E_{\varepsilon, 10^{-j}}^N \}.
\end{equation}
Finally, the parameter-uniform error over the full range of $\varepsilon$ values is given by $E^N = \max \{  E_{10^{-3}}^N, E_{10^{-4}}^N, \ldots, E_{10^{-14}}^N \}.$
Tables~\ref{ET: example 1 s-mesh} and~\ref{ET: example 1 bs-mesh} list the computed values of $E_\varepsilon^N$ (as given in \eqref{Eq: E eps N}) for various $N$ and $\varepsilon$ on the S~--~mesh and BS~--~mesh, respectively. The last row of each table highlights the maximum error in each column. These results show that the error is robust with respect to $\varepsilon$ and $\mu$, and is converging to zero as $N$ increases. This behavior is further confirmed in Figure~\ref{Fig: example 1 n eps plot}, which shows the maximum pointwise errors plotted against $N$ and $\varepsilon$, indicating that the error steadily decreases on both the S~--~mesh and BS~--~mesh (Figures~\ref{Fig: example 1 s mesh n eps plot} and~\ref{Fig: example 1 bs mesh n eps plot}).

To estimate the rate of convergence, we consider the two-mesh difference $D_{\varepsilon,\mu}^N = \| \vec{Y}^N - \vec{Y}_h^{2N} \|_{\Omega^N},$
where $\vec{Y}_h^{2N}$ is the solution's piecewise linear interpolant on the finer mesh with $2N$ intervals. The parameter-uniform two-mesh error is defined as $D^N = \max_{\varepsilon,\mu} \{ D_{\varepsilon,\mu}^N \},$
with $\varepsilon$ and $\mu$ varying over ${10^{-3}, \ldots, 10^{-14}}$. The rate of convergence is given by
\begin{equation*}
	p^N = \log_2 \left( \dfrac{D^N}{D^{2N}} \right) 
\end{equation*}
and the computed parameter-uniform order of convergence is
\begin{equation}  \label{Eq: p*}
	p^* = \min_N p^N.
\end{equation}

Tables~\ref{RCT: example 1 s-mesh} and~\ref{RCT: example 1 bs-mesh} present the computed values of $D^N$ and $p^N$ corresponding to the S~--~mesh and BS~--~mesh, respectively. The results demonstrate a convergence of order $N^{-2} \ln^2 N$ for the S~--~mesh and $N^{-2}$ for the BS~--~mesh. Based on equation~\eqref{Eq: p*}, the parameter-uniform order of convergence for Example~\ref{E: example 1} is computed as $p^* = 1.051$ for the S~--~mesh. In comparison, the BS~--~mesh yields $p^* = 1.633$.



\begin{table}[!htbp]
	\centering
		\caption{Errors $E_\varepsilon^N$ in the numerical solution of Example \upshape{\ref{E: example 1}} (S~--~mesh)}
		\renewcommand{\arraystretch}{1.2}
		\begin{tabular}{c c c c c c c c}
			\toprule
			\multicolumn{1}{c}{$\varepsilon / N$} & 
			\multicolumn{1}{c}{$2^6$}             & 
			\multicolumn{1}{c}{$2^7$}             & 
			\multicolumn{1}{c}{$2^8$}             & 
			\multicolumn{1}{c}{$2^9$}             & 
			\multicolumn{1}{c}{$2^{10}$}          & 
			\multicolumn{1}{c}{$2^{11}$}          & 
			\multicolumn{1}{c}{$2^{12}$}          \\
			\midrule
			$10^{-3}$    &  1.086e$-$02  &  7.466e$-$03  &  2.569e$-$03  &  8.389e$-$04  &  2.599e$-$04  &  7.909e$-$05  &  2.354e$-$05  \\
			$10^{-4}$    &  2.537e$-$02  &  1.511e$-$02  &  9.072e$-$03  &  3.232e$-$03  &  1.058e$-$03  &  3.269e$-$04  &  9.768e$-$05  \\
			$10^{-5}$    &  7.498e$-$02  &  3.456e$-$02  &  1.341e$-$02  &  4.660e$-$03  &  1.520e$-$03  &  4.766e$-$04  &  1.454e$-$04  \\
			\vdots       &  \vdots       &  \vdots       &  \vdots       &  \vdots       &  \vdots       &  \vdots       &  \vdots     \\
			$10^{-14}$   &  7.498e$-$02  &  3.456e$-$02  &  1.341e$-$02  &  4.660e$-$03  &  1.520e$-$03  &  4.766e$-$04  &  1.454e$-$04  \\
			\midrule
			$E^N$        &  7.498e$-$02  &  3.456e$-$02  &  1.341e$-$02  &  4.660e$-$03  &  1.520e$-$03  &  4.766e$-$04  &  1.454e$-$04 \\
			\noalign{\vskip 1pt}
			\botrule
		\end{tabular}
		\label{ET: example 1 s-mesh}
		\centering
		\noindent \footnotesize \textit{CPU time for generating Table~\upshape{\ref{ET: example 1 s-mesh}} in MATLAB R2025a: 145.514907 seconds}
		\end{table}
		
	\begin{table}[!htbp]
		\centering
		\caption{Rates of convergence for Example~\upshape{\ref{E: example 1}} (S~--~mesh)}
		\renewcommand{\arraystretch}{1.2}
		\begin{tabular}{c c c c c c c c}
			\toprule
			\multicolumn{1}{c}{$N$} & 
			\multicolumn{1}{c}{$2^6$}             & 
			\multicolumn{1}{c}{$2^7$}             & 
			\multicolumn{1}{c}{$2^8$}             & 
			\multicolumn{1}{c}{$2^9$}             & 
			\multicolumn{1}{c}{$2^{10}$}          & 
			\multicolumn{1}{c}{$2^{11}$}          & 
			\multicolumn{1}{c}{$2^{12}$}          \\
			\midrule
			$D^N$        	& 5.405e$-$02 & 2.609e$-$02 & 1.034e$-$02 & 3.625e$-$03 & 1.186e$-$03 & 3.722e$-$04 & 1.136e$-$04 \\	
			$p^N$        	& 1.051       & 1.335       & 1.513       & 1.612       & 1.672       & 1.713       &  \\	
			\noalign{\vskip 1pt}
			\botrule
		\end{tabular}
		\label{RCT: example 1 s-mesh}
				\centering
		\noindent \footnotesize \textit{CPU time for generating Table~\upshape{\ref{RCT: example 1 s-mesh}} in MATLAB R2025a: 32.798053 seconds}
	
\end{table}



\begin{table}[!htbp]
	\centering
		\caption{Errors $E_\varepsilon^N$ in the numerical solution of Example \upshape{\ref{E: example 1}} (BS~--~mesh)}
		\renewcommand{\arraystretch}{1.2}
		\begin{tabular}{c c c c c c c c}
			\toprule
			\multicolumn{1}{c}{$\varepsilon / N$} & 
			\multicolumn{1}{c}{$2^6$}             & 
			\multicolumn{1}{c}{$2^7$}             & 
			\multicolumn{1}{c}{$2^8$}             & 
			\multicolumn{1}{c}{$2^9$}             & 
			\multicolumn{1}{c}{$2^{10}$}          & 
			\multicolumn{1}{c}{$2^{11}$}          & 
			\multicolumn{1}{c}{$2^{12}$}          \\
			\midrule
			$10^{-3}$    &  1.273e$-$02  &  5.999e$-$03  &  1.693e$-$03  &  4.492e$-$04  &  1.169e$-$04  &  2.982e$-$05  &  7.557e$-$06 \\
			$10^{-4}$    &  3.395e$-$02  &  1.070e$-$02  &  3.053e$-$03  &  8.271e$-$04  &  2.170e$-$04  &  5.552e$-$05  &  1.384e$-$05  \\
			$10^{-5}$    &  3.395e$-$02  &  1.070e$-$02  &  3.053e$-$03  &  8.271e$-$04  &  2.170e$-$04  &  5.551e$-$05  &  1.384e$-$05  \\
			\vdots       &  \vdots       &  \vdots       &  \vdots       &  \vdots       &  \vdots       &  \vdots       &  \vdots     \\
			$10^{-14}$   &  3.395e$-$02  &  1.070e$-$02  &  3.053e$-$03  &  8.271e$-$04  &  2.170e$-$04  &  5.551e$-$05  &  1.384e$-$05  \\
			\midrule
			$E^N$        &  3.395e$-$02  &  1.070e$-$02  &  3.053e$-$03  &  8.271e$-$04  &  2.170e$-$04  &  5.551e$-$05  &  1.384e$-$05  \\
			\noalign{\vskip 1pt}
			\botrule
		\end{tabular}
		\label{ET: example 1 bs-mesh}
		\centering
		\noindent \footnotesize \textit{CPU time for generating Table~\upshape{\ref{ET: example 1 bs-mesh}} in MATLAB R2025a: 161.229879 seconds}
\end{table}


	\begin{table}
		\centering
		\caption{Rates of convergence for Example~\upshape{\ref{E: example 1}} (BS~--~mesh)}
		\renewcommand{\arraystretch}{1.2}
		\begin{tabular}{c c c c c c c c}
			\toprule
			\multicolumn{1}{c}{$N$} & 
			\multicolumn{1}{c}{$2^6$}             & 
			\multicolumn{1}{c}{$2^7$}             & 
			\multicolumn{1}{c}{$2^8$}             & 
			\multicolumn{1}{c}{$2^9$}             & 
			\multicolumn{1}{c}{$2^{10}$}          & 
			\multicolumn{1}{c}{$2^{11}$}          & 
			\multicolumn{1}{c}{$2^{12}$}          \\
			\midrule
			$D^N$  	& 2.568e$-$02  &  8.280e$-$03  &  2.379e$-$03  &  6.457e$-$04  &  1.695e$-$04  &  4.337e$-$05  &  1.082e$-$05  \\	
			$p^N$  	& 1.633  & 1.799  & 1.881  & 1.930  & 1.967  & 2.004  &        \\	
			\noalign{\vskip 1pt}
			\botrule
		\end{tabular}
		\label{RCT: example 1 bs-mesh}
				\centering
		\noindent \footnotesize \textit{CPU time for generating Table~\upshape{\ref{RCT: example 1 bs-mesh}} in MATLAB R2025a: 35.788285 seconds}
	
\end{table}


\begin{example} \label{E: example 2}
	\begin{align*}
		-\varepsilon^2 y_1'' + 2(x+1)^2y_1 - (1+x^3) y_2 & = 2\exp(x) \\
		-\mu^2 y_2'' - 2 \cos \left( \dfrac{\pi x}{4} \right)y_1 + 2.2 \exp(1-x)y_2 &= 10x+1; 
	\end{align*} %
	with the Robin boundary conditions %
	\begin{alignat*}{2}
		y_1(0) - {\varepsilon} y_1'(0) &= 0; \qquad
		2y_1(1) + {\varepsilon} y_1'(1) &&= 1; \\
		y_2(0) - 3{\mu} y_2'(0) &= 0; \qquad
		y_2(1) + {\mu} y_2'(1) &&= 1.
	\end{alignat*}
\end{example}%

The computed values of $E_\varepsilon^N$ for Example~\ref{E: example 2} are listed in Tables~\ref{ET: example 2 s-mesh} and~\ref{ET: example 2 bs-mesh} for the S~--~mesh and BS~--~mesh, respectively. The values of $E_\varepsilon^N$ are plotted in Figure~\ref{Fig: example 2 n eps plot}, with separate plots for the S~--~mesh in Figure~\ref{Fig: example 2 s mesh n eps plot} and for the BS~--~mesh in  Figure~\ref{Fig: example 2 bs mesh n eps plot}.  The computed parameter-uniform order of convergence on the S-mesh is $p^* = 1.146$ as shown in Table~\ref{RCT: example 2 s-mesh}. On the BS~--~mesh, the value is $p^* = 1.607$ as given in Table~\ref{RCT: example 2 bs-mesh}. 

Examples~\ref{E: example 1} and~\ref{E: example 2} show that the BS~--~mesh yields smaller parameter-uniform errors than the S~--~mesh (refer Tables~\ref{ET: example 1 s-mesh} and \ref{ET: example 1 bs-mesh}). This confirms that the newly constructed BS~--~mesh performs better, with a higher order of convergence-a common principle used to compare numerical methods. To support this visually, log~--~log plots for the S~--~mesh (Figure~\ref{Fig: s mesh loglog plot}) and BS~--~mesh (Figure~\ref{Fig: bs mesh loglog plot}) are shown in Figure~\ref{Fig: loglog plot}. These plots clearly illustrate that the numerical results match the theoretical prediction given in Theorem~\ref{Th: main theorem}.



\begin{table}[!htbp]
	\centering
		\caption{Errors $E_\varepsilon^N$ in the numerical solution of Example \upshape{\ref{E: example 2}} (S~--~mesh)}
		\renewcommand{\arraystretch}{1.2}
		\begin{tabular}{c c c c c c c c}
			\toprule
			\multicolumn{1}{c}{$\varepsilon / N$} & 
			\multicolumn{1}{c}{$2^6$}             & 
			\multicolumn{1}{c}{$2^7$}             & 
			\multicolumn{1}{c}{$2^8$}             & 
			\multicolumn{1}{c}{$2^9$}             & 
			\multicolumn{1}{c}{$2^{10}$}          & 
			\multicolumn{1}{c}{$2^{11}$}          & 
			\multicolumn{1}{c}{$2^{12}$}          \\
			\midrule
			$10^{-3}$    &  2.381e$-$02  &  2.530e$-$02  &  8.502e$-$03  &  2.699e$-$03  &  8.380e$-$04  &  2.537e$-$04  &  7.549e$-$05  \\
			$10^{-4}$    &  1.437e$-$01  &  4.657e$-$02  &  1.278e$-$02  &  3.573e$-$03  &  1.170e$-$03  &  3.613e$-$04  &  1.079e$-$04  \\
			$10^{-5}$    &  2.389e$-$01  &  1.046e$-$01  &  3.960e$-$02  &  1.372e$-$02  &  4.508e$-$03  &  1.429e$-$03  &  4.411e$-$04  \\
			$10^{-6}$    &  2.389e$-$01  &  1.046e$-$01  &  3.961e$-$02  &  1.373e$-$02  &  4.509e$-$03  &  1.429e$-$03  &  4.412e$-$04  \\
			\vdots       &  \vdots       &  \vdots       &  \vdots       &  \vdots       &  \vdots       &  \vdots       &  \vdots     \\
			$10^{-14}$   &  2.389e$-$01  &  1.046e$-$01  &  3.961e$-$02  &  1.373e$-$02  &  4.509e$-$03  &  1.429e$-$03  &  4.412e$-$04  \\
			\midrule
			$E^N$        &  2.389e$-$01  &  1.046e$-$01  &  3.961e$-$02  &  1.373e$-$02  &  4.509e$-$03  &  1.429e$-$03  &  4.412e$-$04 \\
			\noalign{\vskip 1pt}
			\botrule
		\end{tabular}
		\label{ET: example 2 s-mesh}
	 	\centering
		\noindent \footnotesize \textit{CPU time for generating Table~\upshape{\ref{ET: example 2 s-mesh}} in MATLAB R2025a: 144.514505 seconds}
\end{table}

	\begin{table}
			\centering
		\caption{Rates of convergence for Example~\upshape{\ref{E: example 2}} (S~--~mesh)}
		\renewcommand{\arraystretch}{1.2}
		\begin{tabular}{c c c c c c c c}
			\toprule
			\multicolumn{1}{c}{$N$} & 
			\multicolumn{1}{c}{$2^6$}             & 
			\multicolumn{1}{c}{$2^7$}             & 
			\multicolumn{1}{c}{$2^8$}             & 
			\multicolumn{1}{c}{$2^9$}             & 
			\multicolumn{1}{c}{$2^{10}$}          & 
			\multicolumn{1}{c}{$2^{11}$}          & 
			\multicolumn{1}{c}{$2^{12}$}          \\
			\midrule
			$D^N$  	& 1.770e$-$01  &  7.997e$-$02  &  3.071e$-$02  &  1.070e$-$02  &  3.520e$-$03  &  1.116e$-$03  &  3.447e$-$04  \\	
			$p^N$   & 1.146  & 1.381  & 1.521  & 1.604  & 1.657  & 1.696       &  \\	
			\noalign{\vskip 1pt}
			\botrule
		\end{tabular}
		\label{RCT: example 2 s-mesh}
		\centering
		\noindent \footnotesize \textit{CPU time for generating Table~\upshape{\ref{RCT: example 2 s-mesh}} in MATLAB R2025a: 31.731012 seconds}

\end{table}



\begin{table}[!htbp]

		\centering
		\caption{Errors $E_\varepsilon^N$ in the numerical solution of Example \upshape{\ref{E: example 2}} (BS~--~mesh)}
		\renewcommand{\arraystretch}{1.2}
		\begin{tabular}{c c c c c c c c}
			\toprule
			\multicolumn{1}{c}{$\varepsilon / N$} & 
			\multicolumn{1}{c}{$2^6$}             & 
			\multicolumn{1}{c}{$2^7$}             & 
			\multicolumn{1}{c}{$2^8$}             & 
			\multicolumn{1}{c}{$2^9$}             & 
			\multicolumn{1}{c}{$2^{10}$}          & 
			\multicolumn{1}{c}{$2^{11}$}          & 
			\multicolumn{1}{c}{$2^{12}$}          \\
			\midrule
			$10^{-3}$    &  2.982e$-$02  &  1.105e$-$02  &  3.657e$-$03  &  9.950e$-$04  &  2.622e$-$04  &  8.084e$-$05  &  2.539e$-$05  \\
			$10^{-4}$    &  2.195e$-$01  &  7.002e$-$02  &  1.996e$-$02  &  5.376e$-$03  &  1.400e$-$03  &  3.555e$-$04  &  8.799e$-$05  \\
			$10^{-5}$    &  2.197e$-$01  &  7.010e$-$02  &  1.998e$-$02  &  5.381e$-$03  &  1.401e$-$03  &  3.557e$-$04  &  8.802e$-$05  \\
			$10^{-6}$    &  2.197e$-$01  &  7.011e$-$02  &  1.998e$-$02  &  5.381e$-$03  &  1.401e$-$03  &  3.557e$-$04  &  8.803e$-$05  \\
			\vdots       &  \vdots       &  \vdots       &  \vdots       &  \vdots       &  \vdots       &  \vdots       &  \vdots     \\
			$10^{-14}$   &  2.197e$-$01  &  7.011e$-$02  &  1.998e$-$02  &  5.381e$-$03  &  1.401e$-$03  &  3.558e$-$04  &  8.803e$-$05  \\
			\midrule
			$E^N$    &  2.197e$-$01  &  7.011e$-$02  &  1.998e$-$02  &  5.381e$-$03  &  1.401e$-$03  &  3.558e$-$04  &  8.803e$-$05  \\
			\noalign{\vskip 1pt}
			\botrule
		\end{tabular}
		\label{ET: example 2 bs-mesh}
		\centering
		\noindent \footnotesize \textit{CPU time for generating Table~\upshape{\ref{ET: example 2 bs-mesh}} in MATLAB R2025a: 158.558428 seconds}
\end{table}
	
	\begin{table}
		\centering
		\caption{Rates of convergence for Example~\upshape{\ref{E: example 2}} (BS~--~mesh)}
		\renewcommand{\arraystretch}{1.2}
		\begin{tabular}{c c c c c c c c}
			\toprule
			\multicolumn{1}{c}{$N$} & 
			\multicolumn{1}{c}{$2^6$}             & 
			\multicolumn{1}{c}{$2^7$}             & 
			\multicolumn{1}{c}{$2^8$}             & 
			\multicolumn{1}{c}{$2^9$}             & 
			\multicolumn{1}{c}{$2^{10}$}          & 
			\multicolumn{1}{c}{$2^{11}$}          & 
			\multicolumn{1}{c}{$2^{12}$}          \\
			\midrule
			$D^N$  	& 1.649e$-$01  &  5.412e$-$02  &  1.556e$-$02  &  4.200e$-$03  &  1.095e$-$03  &  2.779e$-$04  &  6.877e$-$05  \\	
			$p^N$        	& 1.607  & 1.798  & 1.889  & 1.940  & 1.978  & 2.015       &  \\	
			\noalign{\vskip 1pt}
			\botrule
		\end{tabular}
		\label{RCT: example 2 bs-mesh}
				\centering
		\noindent \footnotesize \textit{CPU time for generating Table~\upshape{\ref{RCT: example 2 bs-mesh}} in MATLAB R2025a: 35.714245 seconds}
	
\end{table}


\begin{figure}[!htbp]
	\centering
	
	\begin{subfigure}[b]{0.49\textwidth}
		\centering
		\includegraphics[width=\textwidth]{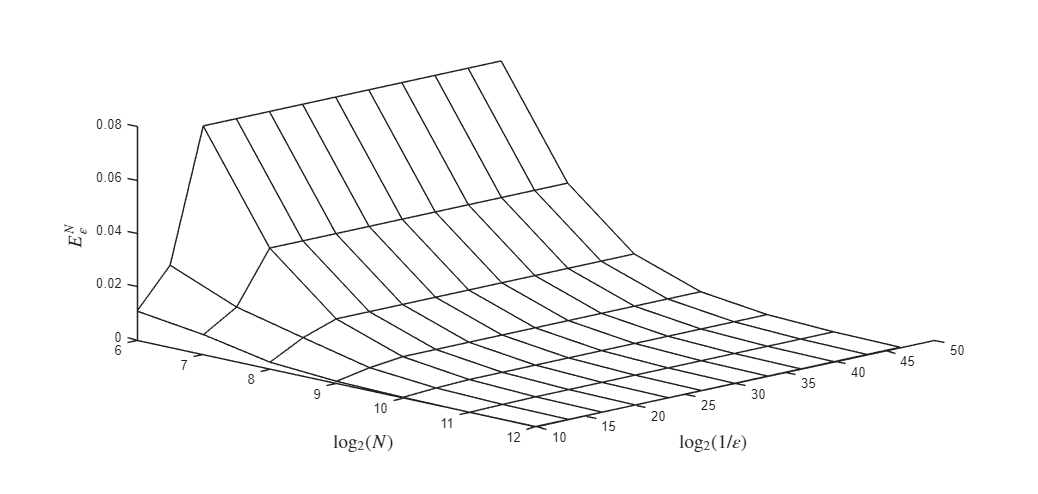}
		\caption{On S--mesh}
		\label{Fig: example 1 s mesh n eps plot}
	\end{subfigure}
	\hfill
	\begin{subfigure}[b]{0.49\textwidth}
		\centering
		\includegraphics[width=\textwidth]{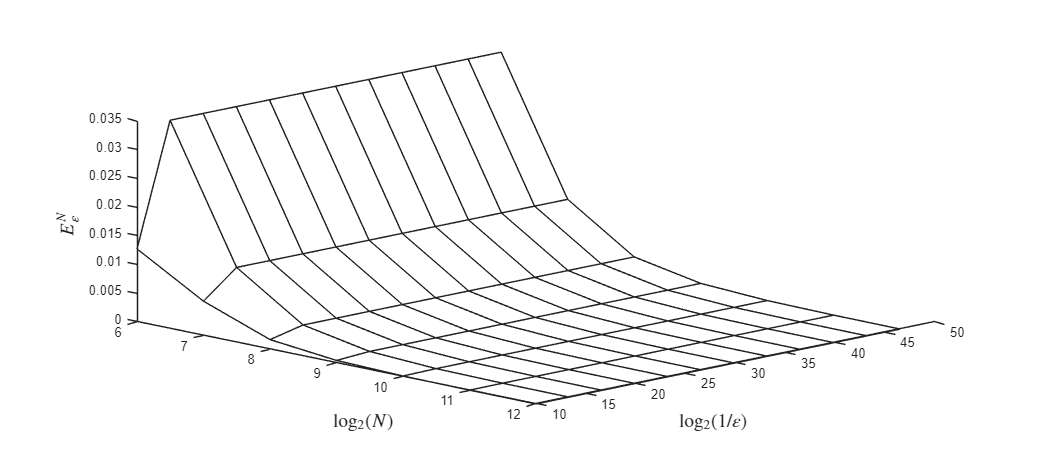}
		\caption{On BS--mesh}
		\label{Fig: example 1 bs mesh n eps plot}
	\end{subfigure}
	\caption{Maximum pointwise errors $E_\varepsilon^N$ vs. $N$ and $\varepsilon$ for Example~\ref{E: example 1}}
	\label{Fig: example 1 n eps plot}
	
	\vspace{1em} 
	
	\begin{subfigure}[b]{0.49\textwidth}
		\centering
		\includegraphics[width=\textwidth]{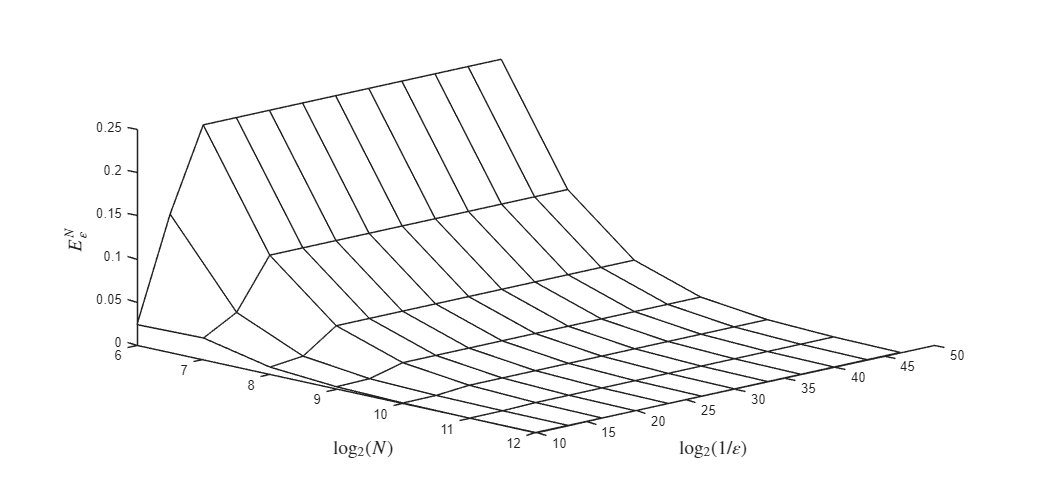}
		\caption{On S--mesh}
		\label{Fig: example 2 s mesh n eps plot}
	\end{subfigure}
	\hfill
	\begin{subfigure}[b]{0.49\textwidth}
		\centering
		\includegraphics[width=\textwidth]{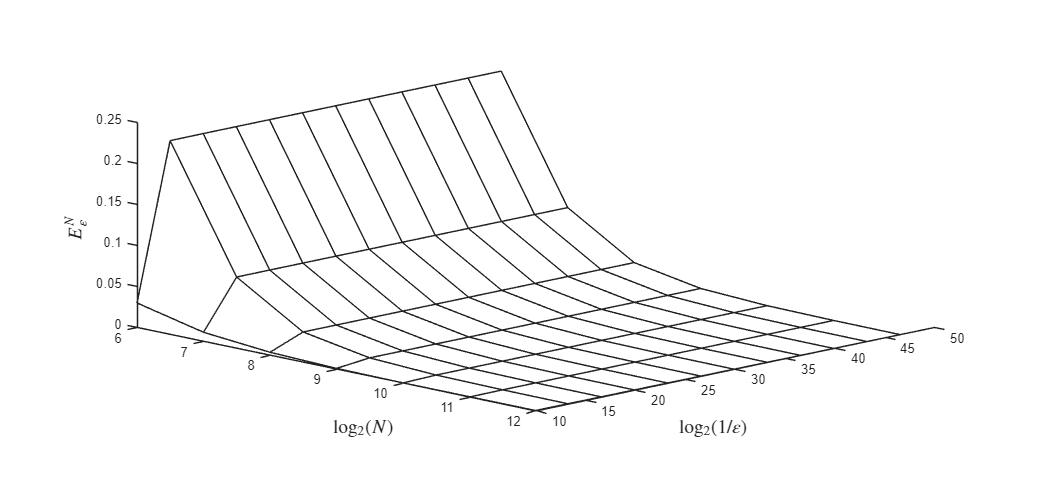}
		\caption{On BS--mesh}
		\label{Fig: example 2 bs mesh n eps plot}
	\end{subfigure}
	\caption{Maximum pointwise errors $E_\varepsilon^N$ vs. $N$ and $\varepsilon$ for Example~\ref{E: example 2}}
	\label{Fig: example 2 n eps plot}
	
\end{figure}

\begin{figure}[!htbp]
	\centering
	
	\begin{subfigure}[b]{0.45\textwidth}
		\centering
		\includegraphics[width=\textwidth]{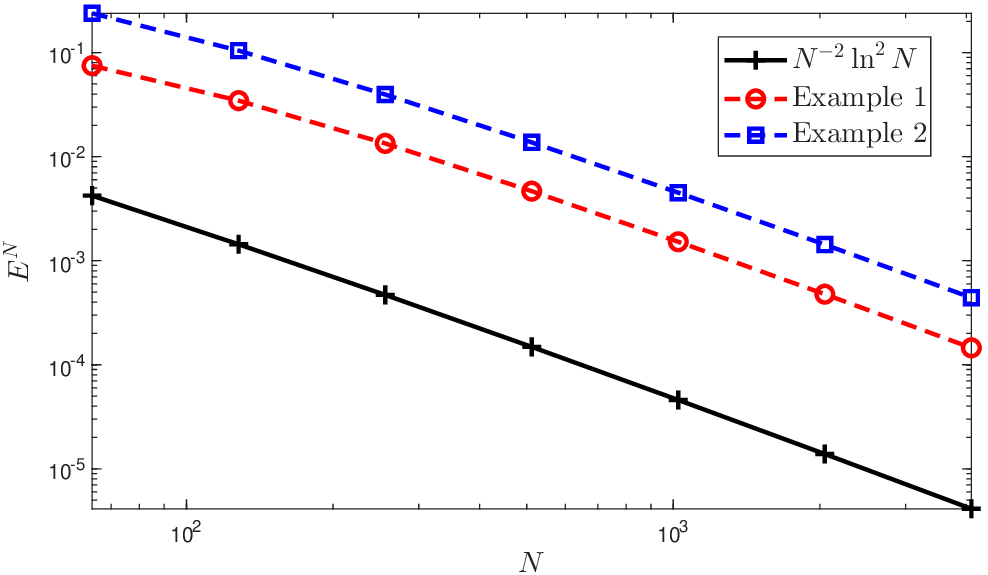}
		\caption{On S-mesh}
		\label{Fig: s mesh loglog plot}
	\end{subfigure}
	\hfill
	\begin{subfigure}[b]{0.45\textwidth}
		\centering
		\includegraphics[width=\textwidth]{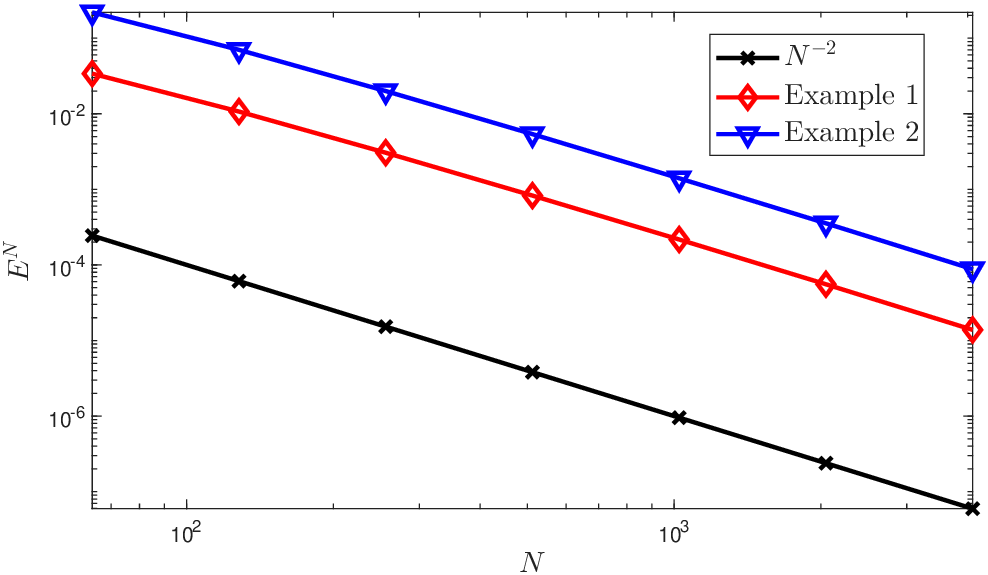}
		\caption{On BS-mesh}
		\label{Fig: bs mesh loglog plot}
	\end{subfigure}
	
	\caption{Loglog plots of $E^N$ vs.\ $N$}
	\label{Fig: loglog plot}
\end{figure}

\section*{Conclusion}
In this work, we considered a weakly-coupled system of singularly perturbed problems of reaction-diffusion-type with Robin boundary conditions, where the leading terms are multiplied by small positive parameters that may vary in magnitude. The numerical solution was obtained using a piecewise-uniform Shishkin mesh and a modified Bakhvalov-Shishkin (BS) mesh. We carried out a detailed truncation error analysis and established the stability of the method. Theoretical results show that the scheme achieves exact second-order convergence on the BS mesh and nearly second-order convergence on the Shishkin mesh. To support these findings, two numerical experiments were conducted, confirming the parameter-uniform convergence of the method. The results clearly demonstrate that the proposed BS mesh yields more accurate solutions than the standard S-mesh for the same numerical scheme.

\backmatter


\section*{Data availability}
No data was utilized for the research work done in this article.


\section*{Acknowledgements}
The first author expresses gratitude to the Ministry of Education (MoE), Govt. of India for the financial support.


\section*{CRediT authorship contribution statement} 

\textit{Kousalya Ramanujam:} Conceptualization, Methodology, Software, Formal analysis, Investigation, Writing - original draft. \textit{Vembu Shanthi:} Conceptualization, Methodology, Formal analysis, Investigation, Supervision, Writing - review \& editing.


\section*{Declaration of competing interest} 

The authors have no competing interests to declare that are relevant to the content of this article. \\

\bibliographystyle{sn-mathphys-num}

\end{document}